\documentclass{svjour3}                     
\smartqed  
\usepackage{graphicx}
\usepackage{amssymb}
\usepackage{amsfonts}
\usepackage{amsmath}
\usepackage{hyperref}
\usepackage{colortbl}
\newtheorem{de}{Definition}
\newtheorem{thm}{Theorem}
\newtheorem{lem}{Lemma}
\newtheorem{cor}{Corollary}
\newtheorem{prop}{Proposition}

\newcommand{\beqn}{\begin{equation}}
\newcommand{\eeqn}{\end{equation}}
\newcommand{\barr}{\begin{array}}
\newcommand{\earr}{\end{array}}
\newcommand{\beqa}{\begin{eqnarray*}}
\newcommand{\eeqa}{\end{eqnarray*}}

\begin{document}

\title{Zeros of Orthogonal Polynomials Generated by Canonical Perturbations on Standard Measure 
\thanks{This work has been done in the framework of a joint project of a joint project of %
Direcci\'on General de Investigaci\'on, Ministerio de Educaci\'on y Ciencia
of Spain and the Brazilian Science Foundation CAPES, Project CAPES/DGU 160/08. 
The work of the first and second authors has been supported by Direcci\'on %
General de Investigaci\'on, Ministerio de Ciencia e Innovaci\'on of Spain, grant %
MTM2009-12740-C03-01. The work of the third author has been supported by FAPESP under %
grant 2009/14776-5.}}


\author{Edmundo J. Huertas \and Francisco Marcell\'an \and Fernando R. Rafaeli.}
\institute{Edmundo J. Huertas \at
              Departamento de Matem\'aticas, Escuela Polit\'ecnica Superior \\
              Universidad Carlos III, Legan\'es-Madrid, Spain\\
              \email{ehuertas@math.uc3m.es}           
           \and
           Francisco Marcell\'an \at
              Departamento de Matem\'aticas, Escuela Polit\'ecnica Superior  \\
              Universidad Carlos III, Legan\'es-Madrid, Spain\\
              \email{pacomarc@ing.uc3m.es}           
              second address
           \and
           Fernando R. Rafaeli \at
              Departamento de Ci\^encias de Computa\c{c}\~ao e Estat\'{\i}stica \\
              IBILCE, Universidade Estadual Paulista, Brazil \\
              \email{fe\_ro\_rafaeli@yahoo.com.br}           
}

\date{Received: date / Accepted: date}

\maketitle

\begin{abstract}
In this contribution, the behavior of zeros of orthogonal polynomials
associated with canonical linear spectral transformations of measures
supported in the real line is analyzed. An electrostatic interpretation of
them is given.
\end{abstract}

\section{Introduction}

\subsection{Basic theory of orthogonal polynomials on the real line}

Let $\mu $ be a positive Borel measure supported on a subset $\Sigma $ of
the real line with infinitely many points and such that 
\begin{equation*}
\int_{\Sigma }|x|^{n}d\mu (x)<\infty ,
\end{equation*}%
for every $n\in \mathbb{Z}_{+}$.

Given a measure $\mu $, we define the standard inner product $\langle \cdot
,\cdot \rangle _{\mu }:\mathbb{P}\times \mathbb{P}\rightarrow \mathbb{R}$ by 
\begin{equation}
\langle p,q\rangle _{\mu }=\int_{\Sigma }p(x)q(x)d\mu (x),\quad p,q\in 
\mathbb{P},  \label{inner-prod-1}
\end{equation}%
where $\mathbb{P}$ is the linear space of the polynomials with real
coefficients, and the corresponding norm $\Vert \cdot \Vert _{\mu }:\mathbb{P%
}\rightarrow \lbrack 0,\infty )$ is given, as usual, by 
\begin{equation}
\Vert p\Vert _{\mu }=\sqrt{\int_{\Sigma }|p(x)|^{2}d\mu (x)},\quad p\in 
\mathbb{P}.  \label{norm-1}
\end{equation}

\setcounter{de}{0}

\begin{de}
Let $\{p_{n}(x)\}_{n\geq 0}$ be a sequence of polynomials such that

\begin{itemize}
\item[1.] $p_{n}(x)$ is a polynomial with degree exactly $n$.

\item[2.] $\langle p_{n},p_{m}\rangle _{\mu }=0$, for $m\neq n$.
\end{itemize}

\noindent $\{p_{n}(x)\}_{n\geq 0}$ is said to be a \textbf{sequence of
standard orthogonal polynomials.} If the leading coefficient of $p_{n} $ is $%
1$, then the sequence is said to be a standard \textbf{monic orthogonal
polynomial sequence (MOPS, in short).}
\end{de}

\setcounter{prop}{0}

\begin{prop}
Each positive Borel measure $\mu $ determines uniquely a standard MOPS.
\end{prop}

A MOPS is generated by a three term recurrence relation. It will play an
important role in the sequel.

\setcounter{prop}{1}

\begin{prop}[Three Term Recurrence Relation TTRR]
Let $\{p_{n}(x)\}_{n\geq 0}$ be a MOPS. They satisfy a three-term recurrence
relation 
\begin{equation}
p_{n+1}(x)=(x-\beta _{n})p_{n}(x)-\gamma _{n}p_{n-1}(x),\quad n\geq 0,
\label{Recurrence}
\end{equation}%
with initial conditions $p_{0}(x)=1$ and $p_{-1}(x)=0$. The recurrence
coefficients are given by 
\begin{equation*}
\beta _{n}=\dfrac{\langle xp_{n},p_{n}\rangle _{\mu }}{\Vert p_{n}\Vert
_{\mu }^{2}},\ n\geq 0,
\end{equation*}%
and 
\begin{equation*}
\gamma _{n}=\dfrac{\Vert p_{n}\Vert _{\mu }^{2}}{\Vert p_{n-1}\Vert _{\mu
}^{2}}>0, n\geq 1.
\end{equation*}
\end{prop}

Next, we will introduce the $nth$ Kernel associated with the MOPS $%
\{p_{n}(x)\}_{n\geq 0}.$ It satisfies a reproducing property for every
polynomial of degree at most $n$ as well it can be represented in a simple
way in terms of the polynomials $p_{n}(x)$ and $p_{n+1}(x)$ throught the
Christoffel-Darboux Formula, that can be deduced in a straightforward way
from the three term recurrence relation (see \cite{Chi78}). %
\setcounter{prop}{2}

\begin{prop}
Let $\{p_{n}(x)\}_{n\geq 0}$ be a MOPS. If we denote the $n$th Kernel
polynomial by 
\begin{equation}  \label{Kernel}
K_{n}(x,y)=\sum_{j=0}^{n}\dfrac{p_{j}(y)p_{j}(x)}{\Vert p_{j}\Vert _{\mu
}^{2}},
\end{equation}%
then, for every $n\in \mathbb{N}$, 
\begin{equation*}
K_{n}(x,y)=\dfrac{1}{\Vert p_{n}\Vert _{\mu }^{2}}\dfrac{%
p_{n+1}(x)p_{n}(y)-p_{n}(x)p_{n+1}(y)}{x-y}.
\end{equation*}%
and, if $x=y$ we have the so called confluent form 
\begin{equation*}
K_{n}(x,x)=\frac{p_{n+1}^{\prime }(x)p_{n}(x)-p_{n}^{\prime }(x)p_{n+1}(x)}{%
\left\Vert p_{n}\right\Vert _{\mu }^{2}} .  \label{Kernel-n-confl}
\end{equation*}
\end{prop}

The zeros of standard orthogonal polynomials are the nodes of the Gaussian
quadrature rules and they have nice analytic properties. Next, we summarize
some of them. For more information, see \cite{Chi78}, \cite{NiUv88}, and 
\cite{Sze75} \setcounter{prop}{3}

\begin{prop}
\label{PropZeros} Let $\mu $ be a positive Borel measure defined as above
and $\{p_{n}(x)\}_{n\geq 0}$ the corresponding MOPS. Then,

\begin{itemize}
\item[1.] For each $n\geq 1$, the polynomial $p_{n}(x)$ has $n$ real and
simple zeros in the interior of $C_{0}(\Sigma )$, the convex hull of $\Sigma 
$.

\item[2.] The zeros of $p_{n+1}(x)$ interlace with the zeros of $p_{n}(x)$.

\item[3.] Between any two consecutive zeros of $p_{n}(x)$ there is at least
one zero of $p_{m}(x)$, for $m>n\geq 2$.

\item[4.] If $(\alpha,\beta)\subset C_0(\Sigma)$ with $(\alpha,\beta)\cap
\Sigma=\emptyset$ then at most one zero of each polynomial $p_n(x)$ belongs
to $(\alpha,\beta)$.

\item[5.] Each point of $\Sigma $ attracts zeros of the MOPS. In other
words, the zeros of a MOPS are dense in $\Sigma $.
\end{itemize}
\end{prop}

Next, we will analyze the behavior of zeros of polynomial of the form $%
f(x)=h_{n}(x)+cg_{n}(x)$

We need the following lemma concerning the behavior and the asymptotics of
the zeros of linear combinations of two polynomials with interlacing zeros
(see \cite[Lemma 1]{BracDimRanga02} and \cite[Lemma 3]{DimitrovMelloRafaeli}%
).

\setcounter{lem}{0}

\begin{lem}
\label{L1L} Let $h_{n}(x)=a(x-x_{1})\cdots(x-x_{n})$ and $%
g_{n}(x)=b(x-\zeta_{1})\cdots(x-\zeta_{n})$ be polynomials with real and
simple zeros, where $a$ and $b$ are real positive constants.

\begin{itemize}
\item[(i)] If 
\begin{equation*}
\zeta _{1}<x_{1}<\cdots <\zeta _{n}<x_{n},
\end{equation*}%
then, for any real constant $c>0$, the polynomial 
\begin{equation*}
f(x)=h_{n}(x)+cg_{n}(x)
\end{equation*}%
has $n$ real zeros $\eta _{1}<\cdots <\eta _{n}$ which interlace with the
zeros of $h_{n}(x)$ and $g_{n}(x)$ in the following way 
\begin{equation*}
\zeta _{1}<\eta _{1}<x_{1}<\cdots <\zeta _{n}<\eta _{n}<x_{n}.
\end{equation*}%
Moreover, each $\eta _{k}=\eta _{k}(c)$ is a decreasing function of $c$ and,
for each $k=1,\ldots ,n$, 
\begin{equation}
\lim_{c\rightarrow \infty }\eta _{k}=\zeta _{k}\quad \text{and}\quad
\lim_{c\rightarrow \infty }c[\eta _{k}-\zeta _{k}]=\dfrac{-h_{n}(\zeta _{k})%
}{g_{n}^{\prime }(\zeta _{k})}.  \label{LimZeros-i}
\end{equation}

\item[(ii)] If 
\begin{equation*}
x_{1}<\zeta _{1}<\cdots <x_{n}<\zeta _{n},
\end{equation*}%
then, for any positive real constant $c>0$, the polynomial 
\begin{equation*}
f(x)=h_{n}(x)+cg_{n}(x)
\end{equation*}%
has $n$ real zeros $\eta _{1}<\cdots <\eta _{n}$ which interlace with the
zeros of $h_{n}(x)$ and $g_{n}(x)$ as follows 
\begin{equation*}
x_{1}<\eta _{1}<\zeta _{1}<\cdots <x_{n}<\eta _{n}<\zeta _{n}.
\end{equation*}%
Moreover, each $\eta _{k}=\eta _{k}(c)$ is an increasing function of $c$
and, for each $k=1,\ldots ,n$, 
\begin{equation}
\lim_{c\rightarrow \infty }\eta _{k}=\zeta _{k}\quad \text{and}\quad
\lim_{c\rightarrow \infty }c[\zeta _{k}-\eta _{k}]=\dfrac{h_{n}(\zeta _{k})}{%
g_{n}^{\prime }(\zeta _{k})}.  \label{LimZeros-ii}
\end{equation}
\end{itemize}
\end{lem}

In the last years some attention has been paid to the so called canonical
spectral transformations of measures. Some authors have analyzed them from
the point of view of Stieltjes functions associated with such a kind of
perturbations (see \cite{Zhe97})or from the relation between the
corresponding Jacobi matrices (see \cite{Yoon}). Our present contribution is
focused on the behaviour of zeros of MOPS associated with such
transformations of the measures. In particular, we are interested in the
Christoffel and Uvarov transformations which are given as a multiplication
of the measure by a positive linear polynomial in its support and the
addition of a Dirac mass at a point outside the support, respectively.

The structure of the manuscript is as follows. In Section 2 the
representation of the perturbed MOPS in terms of the initial ones is done.
In Section 3 we analyze the behaviour of the zeros of the MOPS when an
Uvarov transform is introduced. In particular, we obtain such a behavior
when the mass $N$ tends to infinity as well as we characterize the values of
the mass $N$ such the smallest (respectively, the largest) zero of these
MOPS is located outside the support of the measure. In Section 4, we check
these results in the cases of the Jacobi-type and Laguerre-type orthogogonal
polynomials introduced by T. H. Koornwinder (\cite{Koornwinder84}). Section
5 is devoted to the electrostatic interpretation of the zero distribution as
equilibrium points in a logarithmic potential interaction under the action
of an external field.We analyze such an equilibrium problem when the mass
point is located on the boundary or in the exterior of the support of the
measure, respectively.

\section{Canonical perturbations of a measure}

Let $\{p_{n}(x)\}_{n\geq 0}$ be the MOPS with respect to a positive Borel
measure $\mu $ defined as above. We will consider some canonical
perturbations of the measure, which are called spectral linear
transformations (see \cite{Zhe97} and \cite{Yoon}). \vspace{0.4cm}


\subsection{Christoffel perturbation}

Let $\{p_{n}^{\ast }(a;x)\}_{n\geq 0}$ be the MOPS associated with the
measure 
\begin{equation*}
d\mu ^{\ast }=(x-a)d\mu ,
\end{equation*}%
with $a\not\in C_{0}(\Sigma )$. This means that $p_{n}(a)\neq 0$ for every $%
n\geq 1$.

The polynomial $p_{n}^{\ast }(a;x)$ is given by (see \cite[(7.3)]{Chi78}) 
\begin{equation}  \label{p1}
p_{n}^{\ast }(a;x)=\dfrac{1}{x-a}\left[ p_{n+1}(x)-\dfrac{p_{n+1}(a)}{%
p_{n}(a)}p_{n}(x)\right] =\dfrac{\Vert p_{n}\Vert _{\mu }^{2}}{p_{n}(a)}%
K_{n}(a,x),
\end{equation}%
i.e., $p_{n}^{\ast }(a;x)$ is a monic kernel polynomial. Notice that $%
p_{n}^{\ast }(a;a)\neq 0$. Let $x_{n,k}$ and $x_{n,k}^{\ast }:=x_{n,k}^{\ast
}(a)$ be the zeros of $p_{n}(x)$ and $p_{n}^{\ast }(a;x)$, respectively, all
arranged in an increasing order, and assume that $C_{0}(\Sigma )=[\xi ,\eta
] $. From (\ref{p1}) and Proposition \ref{PropZeros} (item 2) the following
interlacing property of these zeros holds

\begin{itemize}
\item[(i)] If $a\leq \xi $, then $p_{n+1}(a)/p_{n}(a)<0$, and%
\begin{equation*}
Sign\left[ p_{n}^{\ast }(a;x_{n+1,k})\right] =Sign\left[ p_{n}(x_{n+1,k})%
\right] ,\ k=1,\ldots ,n+1,
\end{equation*}%
as well as 
\begin{equation*}
Sign\left[ p_{n}^{\ast }(a;x_{n,k})\right] =Sign\left[ p_{n+1}(x_{n,k})%
\right] ,\ k=1,\ldots ,n.
\end{equation*}%
Therefore 
\begin{equation}
x_{n+1,1}<x_{n,1}<x_{n,1}^{\ast }<x_{n+1,2}<\cdots <x_{n,n}<x_{n,n}^{\ast
}<x_{n+1,n+1};  \label{interlace-i}
\end{equation}

\item[(ii)] If $a\geq \eta $, then $p_{n+1}(a)/p_{n}(a)>0$ and, as a
consequence, 
\begin{equation*}
Sign\left[ p_{n}^{\ast }(a;x_{n+1,k})\right] =Sign\left[ p_{n}(x_{n+1,k})%
\right] ,\ k=1,\ldots ,n+1,
\end{equation*}%
and 
\begin{equation*}
Sign\left[ p_{n}^{\ast }(a;x_{n,k})\right] =-Sign\left[ p_{n+1}(x_{n,k})%
\right] ,\ k=1,\ldots ,n.
\end{equation*}%
Therefore 
\begin{equation}
x_{n+1,1}<x_{n,1}^{\ast }<x_{n,1}<\cdots <x_{n+1,n}<x_{n,n}^{\ast
}<x_{n,n}<x_{n+1,n+1}.  \label{interlace-ii}
\end{equation}
\end{itemize}


\subsection{Iterated-Christoffel}

Let $\{p_{n}^{\ast \ast }(a;x)\}_{n\geq 0}$ be the MOPS associated with the
measure%
\begin{equation*}
d\mu ^{\ast \ast }=(x-a)^{2}d\mu ,
\end{equation*}%
with $a\not\in C_{0}(\Sigma )$. Using (\ref{p1}) we deduce that 
\begin{eqnarray}
p_{n}^{\ast \ast }(a;x) &=&\dfrac{1}{x-a}\left[ p_{n+1}^{\ast }(a;x)-\dfrac{%
p_{n+1}^{\ast }(a;a)}{p_{n}^{\ast }(a;a)}p_{n}^{\ast }(a;x)\right]  \notag \\
&=&\dfrac{\Vert p_{n}^{\ast }\Vert _{\mu ^{\ast }}^{2}}{p_{n}^{\ast }(a;a)}%
K_{n}^{\ast }(a,x)  \label{p2} \\
&=&\dfrac{1}{(x-a)^{2}}\left[ p_{n+2}(x)-d_{n}p_{n+1}(x)+e_{n}p_{n}(x)\right]
,  \notag
\end{eqnarray}%
where $K_{n}^{\ast }(a,x)$ is the Kernel polynomial associated with $%
p_{n}^{\ast }(a;x)$, and 
\begin{eqnarray*}
d_{n} &=&\dfrac{p_{n+2}(a)}{p_{n+1}(a)}+\dfrac{p_{n+1}^{\ast }(a;a)}{%
p_{n}^{\ast }(a;a)}=\dfrac{p_{n+2}(a)+p_{n}(a)}{p_{n+1}(a)}e_{n}, \\
e_{n} &=&\dfrac{p_{n+1}^{\ast }(a;a)}{p_{n}^{\ast }(a;a)}\dfrac{p_{n+1}(a)}{%
p_{n}(a)}=\dfrac{\Vert p_{n+1}\Vert _{\mu }^{2}}{\Vert p_{n}\Vert _{\mu }^{2}%
}\dfrac{K_{n+1}(a,a)}{K_{n}(a,a)}>0.
\end{eqnarray*}

Notice that $p_{n}^{\ast \ast }(a;a)\neq 0$. Let denote by $x_{n,k}^{\ast
\ast}:=x_{n,k}^{\ast \ast }(a)$ the zeros of $p_{n}^{\ast \ast }(a;x)$,
arranged in an increasing order. Then replacing (\ref{Recurrence}) in (\ref%
{p2}) we obtain 
\begin{equation}
p_{n}^{\ast \ast }(a;x)=\dfrac{1}{(x-a)^{2}}\left[ (x-\beta
_{n+1}-d_{n})p_{n+1}(x)+(e_{n}-\gamma _{n+1})p_{n}(x)\right] .  \label{p3}
\end{equation}%
On the other hand, 
\begin{equation}
e_{n}-\gamma _{n+1}=\dfrac{\Vert p_{n+1}\Vert _{\mu }^{2}}{\Vert p_{n}\Vert
_{\mu }^{2}}\left( \dfrac{K_{n+1}(a,a)}{K_{n}(a,a)}-1\right) >0.
\label{en-gamma}
\end{equation}%
Evaluating $p_{n}^{\ast \ast }(a;x)$ at the zeros $x_{n+1,k}$, from (\ref{p3}%
) and (\ref{en-gamma}), we get 
\begin{equation}
Sign\left[ p_{n}^{\ast \ast }(a;x_{n+1,k})\right] =Sign\left[
p_{n}(x_{n+1,k})\right] ,\ k=1,\ldots ,n+1.  \label{Z2}
\end{equation}%
Thus, from (\ref{Z2}) and Proposition \ref{PropZeros} (item 2) we obtain the
following interlacing property:

\setcounter{thm}{0}

\begin{thm}
\label{T1} The inequalities 
\begin{equation}  \label{EqT1}
x_{n+1,1}<x_{n,1}^{**}<x_{n+1,2}<x_{n,2}^{**}<%
\cdots<x_{n+1,n}<x_{n,n}^{**}<x_{n+1,n+1}
\end{equation}
hold for every $n\in\mathbb{N}$.
\end{thm}

%
%

\subsection{Uvarov perturbation}

Let $\{p_{n}^{N}(a;x)\}_{n\geq 0}$ be the MOPS associated with the measure 
\begin{equation*}
d\mu _{N}=d\mu +N\delta _{a},
\end{equation*}%
with $N\in \mathbb{R}_{+}$, $\delta _{a}$ the Dirac delta function in $x=a$,
and $a\not\in C_{0}(\Sigma )$. R. \'{A}lvarez-Nodarse, F. Marcell\'{a}n and
J. Petronilho \cite[(8)]{WKB96} obtained the following representation for
such polynomials in terms of the MOPS $\{p_{n}(x)\}_{n\geq 0}$ (see also 
\cite{Uv69}.)%
\begin{equation}
p_{n}^{N}(a;x)=p_{n}(x)-\dfrac{Np_{n}(a)}{1+NK_{n-1}(a,a)}K_{n-1}(a,x).
\label{Con-For}
\end{equation}

Next, we give another connection formula for the Uvarov's orthogonal
polynomials $p_{n}^{N}(a;x)$ using the standard orthogonal polynomials $%
p_{n}(x)$ and the Iterated-Christoffel's orthogonal polynomials $p_{n}^{\ast
\ast }(a;x)$.

\setcounter{thm}{1}

\begin{thm}[Connection Formula]
The polynomials $\{\widehat{p}_{n}^{N}(a;x)\}_{n\geq 0}$, with $\widehat{p}%
_{n}^{N}(a;x)=k_{n}p_{n}^{N}(a;x)$, can be represented as 
\begin{equation}
\widehat{p}_{n}^{N}(a;x)=p_{n}(x)+NB_{n}(x-a)p_{n-1}^{\ast \ast }(a;x),
\label{EqT2}
\end{equation}%
with 
\begin{equation}
B_{n}=\frac{-p_{n}(a)}{\langle x-a,p_{n-1}^{\ast \ast }\rangle _{\mu }}%
=K_{n-1}\left( a,a\right) >0  \label{Bn}
\end{equation}%
and $k_{n}=1+NB_{n}$.
\end{thm}

\begin{proof}
In order to prove the orthogonality of the polynomials defined by (\ref{EqT2}%
), we deal with the basis $1,(x-a),(x-a)^{2},\ldots ,(x-a)^{n}$ of the
linear space of polynomials of degree at most $n$. Then, 
\begin{equation*}
\begin{array}{rcl}
\langle 1,\widehat{p}_{n}^{N}\rangle _{\mu _{_{N}}} & = & \langle
1,p_{n}\rangle _{\mu }+NB_{n}\langle 1,(x-a)p_{n-1}^{\ast \ast }\rangle
_{\mu }+Np_{n}(a)=0\vspace{0.3cm} \\ 
\langle (x-a),\widehat{p}_{n}^{N}\rangle _{\mu _{_{N}}} & = & \langle
(x-a),p_{n}\rangle _{\mu }+NB_{n}\langle 1,p_{n-1}^{\ast \ast }\rangle _{\mu
^{\ast \ast }}=0 \\ 
& \vdots &  \\ 
\langle (x-a)^{n-1},\widehat{p}_{n}^{N}\rangle _{\mu _{_{N}}} & = & \langle
(x-a)^{n-1},p_{n}\rangle _{\mu }+NB_{n}\langle (x-a)^{n-2},p_{n-1}^{\ast
\ast }\rangle _{\mu ^{\ast \ast }}=0,%
\end{array}%
\end{equation*}%
and, finally, 
\begin{equation*}
\begin{array}{rcl}
\langle (x-a)^{n},\widehat{p}_{n}^{N}\rangle _{\mu _{_{N}}} & = & \langle
(x-a)^{n},p_{n}\rangle _{\mu }+NB_{n}\langle (x-a)^{n-1},p_{n-1}^{\ast \ast
}\rangle _{\mu ^{\ast \ast }}>0\vspace{0.3cm} \\ 
& = & \Vert p_{n}\Vert _{\mu }^{2}+NB_{n}\Vert p_{n-1}^{\ast \ast }\Vert
_{\mu ^{\ast \ast }}^{2}>0.%
\end{array}%
\end{equation*}

In order to prove (\ref{Bn}), from (\ref{p2}) and (\ref{p1}) we get 
\begin{eqnarray*}
\langle x-a,p_{n-1}^{\ast \ast}\rangle _{\mu } &=&\int\left(
x-a\right)p_{n-1}^{\ast \ast}(a;x)d\mu(x)\vspace{0.3cm} \\
&=&\int\left( x-a\right)\frac{1}{x-a}\left[p_{n}^{\ast}(a;x)-\frac{%
p_{n}^{\ast}(a;a)}{p_{n-1}^{\ast}(a;a)}p_{n-1}^{\ast}(a;x)\right]d\mu(x)%
\vspace{0.3cm} \\
&=&\int p_{n}^{\ast}(a;x)d\mu(x)-\frac{p_{n}^{\ast}(a;a)}{p_{n-1}^{\ast}(a;a)%
}\int p_{n-1}^{\ast}(a;x)d\mu(x)\vspace{0.3cm} \\
&=&\displaystyle\frac{\|p_n\|_{\mu}^{2}}{p_{n}(a)}\int K_{n}(a;x)d\mu(x)%
\vspace{0.3cm} \\
& &-\frac{\|p_n\|_{\mu}^{2}}{p_{n}(a)}\frac{K_{n}(a;a)}{K_{n-1}(a;a)}\frac{%
p_{n-1}(a)}{\|p_{n-1}\|_{\mu}^{2}}\frac{\|p_{n-1}\|_{\mu}^{2}}{p_{n-1}(a)}%
\int K_{n-1}(a;x)d\mu(x)\vspace{0.3cm} \\
&=&\frac{\|p_n\|_{\mu}^{2}}{p_{n}(a)}-\frac{\|p_n\|_{\mu}^{2}}{p_{n}(a)}%
\frac{K_{n}(a;a)}{K_{n-1}(a;a)}\vspace{0.3cm} \\
&=&\frac{\|p_n\|_{\mu}^{2}}{p_{n}(a)}\left(1-\frac{K_{n}(a;a)}{K_{n-1}(a;a)}%
\right). \\
\end{eqnarray*}
Thus 
\begin{equation*}
B_n=\frac{-p_n(a)}{\langle x-a,p_{n-1}^{**}\rangle_{\mu}}= \frac{p_{n}^{2}(a)%
}{\|p_{n}\|_{\mu}^{2}\left[K_{n}(a,a)/K_{n-1}(a,a)-1\right]}=K_{n-1}(a,a)>0
\end{equation*}
\end{proof}

\section{The Zeros}

We call the attention of the reader on the fact that the constant $B_n$
defined as above does not depend on $N$. For this reason, the connection
formula (\ref{EqT2}) is very useful in order to obtain results about
monotonicity, asymptotics, and speed of convergence for the zeros of $%
p_{n}^N(a;x)$ in terms of the mass $N$. Indeed, let assume that $%
x_{n,k}^{N}:=x_{n,k}^{N}(a), k=1, 2, ...,n,$ are the zeros of $p_{n}^N(a;x)$%
. Thus, from (\ref{EqT1}), (\ref{EqT2}), and Lemma \ref{L1L}, we immediately
conclude that

\setcounter{thm}{2}

\begin{thm}
If $C_{0}(\Sigma )=[\xi ,\eta ]$ and $a\leq \xi, $ then 
\begin{equation}
a<x_{n,1}^{N}<x_{n,1}<x_{n-1,1}^{\ast \ast }<x_{n,2}^{N}<x_{n,2}<\cdots
<x_{n-1,n-1}^{\ast \ast }<x_{n,n}^{N}<x_{n,n}.  \label{EqT3-1}
\end{equation}%
Moreover, each $x_{n,k}^{N}$ is a decreasing function of $N$ and, for each $%
k=1,\ldots ,n-1$, 
\begin{equation}
\lim_{N\rightarrow \infty }x_{n,1}^{N}=a,\ \ \lim_{N\rightarrow \infty
}x_{n,k+1}^{N}=x_{n-1,k}^{\ast \ast },  \label{EqT3-2}
\end{equation}%
as well as 
\begin{equation}
\begin{array}{l}
\lim\limits_{N\rightarrow \infty }N[x_{n,1}^{N}-a]=\dfrac{-p_{n}(a)}{%
B_{n}p_{n-1}^{\ast \ast }(a;a)},\vspace{0.3cm} \\ 
\lim\limits_{N\rightarrow \infty }N[x_{n,k+1}^{N}-x_{n-1,k}^{\ast \ast }]=%
\dfrac{-p_{n}(x_{n-1,k}^{\ast \ast })}{B_{n}(x_{n-1,k}^{\ast \ast
}-a)[p_{n-1}^{\ast \ast }(a;x)]_{x=x_{n-1,k}^{\ast \ast }}^{\prime }}.%
\end{array}
\label{LimZeros-1}
\end{equation}
\end{thm}

\setcounter{thm}{3}

\begin{thm}
If $C_{0}(\Sigma )=[\xi ,\eta ]$ and $a\geq \eta, $ then 
\begin{equation}
x_{n,1}<x_{n,1}^{N}<x_{n-1,1}^{\ast \ast }<\cdots
<x_{n,n-1}<x_{n,n-1}^{N}<x_{n-1,n-1}^{\ast \ast }<x_{n,n}<x_{n,n}^{N}<a.
\label{EqT4-1}
\end{equation}%
Moreover, each $x_{n,k}^{N}$ is an increasing function of $N$ and, for each $%
k=1,\ldots ,n-1$, 
\begin{equation}
\lim_{N\rightarrow \infty }x_{n,n}^{N}=a,\ \ \lim_{N\rightarrow \infty
}x_{n,k}^{N}=x_{n-1,k}^{\ast \ast },  \label{EqT4-2}
\end{equation}%
and 
\begin{equation}
\begin{array}{l}
\lim\limits_{N\rightarrow \infty }N[a-x_{n,n}^{N}]=\dfrac{p_{n}(a)}{%
B_{n}p_{n-1}^{\ast \ast }(a;a)},\vspace{0.3cm} \\ 
\lim\limits_{N\rightarrow \infty }N[x_{n-1,k}^{\ast \ast }-x_{n,k}^{N}]=%
\dfrac{p_{n}(x_{n-1,k}^{\ast \ast })}{B_{n}(x_{n-1,k}^{\ast \ast
}-a)[p_{n-1}^{\ast \ast }(a;x)]_{x=x_{n-1,k}^{\ast \ast }}^{\prime }}.%
\end{array}
\label{LimZeros-2}
\end{equation}
\end{thm}

Notice that the mass point $a$ attracts one zero of $p_{n}^{N}(a;x),$ i.e.
when $N\rightarrow \infty $, it captures either the smallest or the largest
zero, according to the location of the point $a$ with respect to the support
of the measure $\mu$.

\subsection{The Minimum Mass}

When either $a<\xi$ or $a>\eta, $ at most one of the zeros of $%
p_{n}^{N}(a;x) $ is located outside of $C_0(\Sigma)=[\xi,\eta]$. In the next
results, we will give explicitly the value $N_0$ of the mass such that for $%
N>N_0$ this situation occurs, i.e, one of the zeros is located outside $%
[\xi,\eta]$. \setcounter{cor}{0}

\begin{cor}
If $C_{0}(\Sigma )=[\xi ,\eta ]$ and $a< \xi ,$ then the smallest zero $%
x_{n,1}^{N}=x_{n,1}^{N}(a)$ satisfies 
\begin{equation*}
\begin{array}{c}
x_{n,1}^{N}>\xi, \ \ \mathrm{for} \ \ N<N_{0}, \vspace{0.4cm} \\ 
x_{n,1}^{N}=\xi, \ \ \mathrm{for} \ \ N=N_{0}, \vspace{0.4cm} \\ 
x_{n,1}^{N}<\xi, \ \ \mathrm{for} \ \ N>N_{0},%
\end{array}%
\end{equation*}
where 
\begin{equation*}
N_0=N_0(n,a,\xi)=\displaystyle\frac{-p_{n}(\xi)}{K_{n-1}\left( a,a\right)
(\xi-a)p_{n-1}^{\ast \ast }(a;\xi)}>0.
\end{equation*}
\end{cor}

\begin{proof}
In order to deduce the location of $x_{n,1}^{N}$ with respect to the point $%
x=\xi$, it is enough to observe that $p_{n}^{N}(\xi)=0$ if and only if $%
N=N_{0}$.
\end{proof}

\setcounter{cor}{1}

\begin{cor}
If $C_{0}(\Sigma )=[\xi ,\eta ]$ and $a>\eta,$ then the largest zero $%
x_{n,n}^{N}=x_{n,n}^{N}(a)$ satisfies 
\begin{equation*}
\begin{array}{c}
x_{n,n}^{N}<\eta, \ \ \mathrm{for} \ \ N<N_{0}, \vspace{0.4cm} \\ 
x_{n,n}^{N}=\eta, \ \ \mathrm{for} \ \ N=N_{0}, \vspace{0.4cm} \\ 
x_{n,n}^{N}>\eta, \ \ \mathrm{for} \ \ N>N_{0},%
\end{array}%
\end{equation*}
where 
\begin{equation*}
N_0=N_0(n,a,\eta)=\displaystyle\frac{-p_{n}(\eta)}{K_{n-1}\left( a,a\right)
(\eta-a)p_{n-1}^{\ast \ast }(a;\eta)}>0.
\end{equation*}
\end{cor}

\begin{proof}
In order to find the location of $x_{n,n}^{N}$ with respect to the point $%
x=\eta$, notice that $p_{n}^{N}(\eta)=0$ if and only if $N=N_{0}$.
\end{proof}

\section{Application to classical measures}

\subsection{Jacobi type (Jacobi-Koornwinder) orthogonal polynomials}

First, we will consider $p_{n}(x)=P_{n}^{\alpha,\beta}(x)$, the classical
monic Jacobi polynomials, which are orthogo\-nal with respect to the measure 
$d\mu _{\alpha ,\beta }=(1-x)^{\alpha }(1+x)^{\beta }dx$, $\alpha ,\beta >-1$%
, supported on $[-1,1]$. We consider the Uvarov perturbations on $\mu
_{\alpha ,\beta }$ with either $a=-1$ or $a=1,$ and $M,N \geq0.$ 
\begin{equation}
d\mu _{M}=d\mu _{\alpha ,\beta }+M\delta _{-1},  \label{Jacobi-type1}
\end{equation}%
\begin{equation}
d\mu _{N}=d\mu _{\alpha ,\beta }+N\delta _{1}.  \label{Jacobi-type2}
\end{equation}%
Such orthogonal polynomials were first studied by T. H. Koornwinder (see 
\cite{Koornwinder84}), in 1984. There, he adds simultaneously two Dirac
delta functions at the end points $x=-1$ and $x=1$, that is, 
\begin{equation*}
d\mu _{M,N}=d\mu _{\alpha ,\beta }+M\delta _{-1}+N\delta _{1}.
\end{equation*}%
Let $\{P_{n}^{\alpha,\beta,M}(x)\}_{n\geq0}$ and $\{P_{n}^{\alpha
,\beta,N}(x)\}_{n\geq0}$ denote the sequences of orthogonal polynomials with
respect (\ref{Jacobi-type1}) and (\ref{Jacobi-type2}), with the
normalization pointed out in Theorem 2, respectively. Then, the connection
formulas are 
\begin{equation*}
P_{n}^{\alpha ,\beta ,M}(x)=P_{n}^{\alpha ,\beta
}(x)+MK_{n-1}(-1,-1)(x+1)P_{n-1}^{\alpha ,\beta +2}(x)
\end{equation*}%
and 
\begin{equation}
P_{n}^{\alpha ,\beta ,N}(x)=P_{n}^{\alpha ,\beta
}(x)+NK_{n-1}(1,1)(x-1)P_{n-1}^{\alpha +2,\beta}(x).  \label{Jac-conexion}
\end{equation}%
It is straightforward to see that 
\begin{equation*}
K_{n-1}(-1,-1)=\frac{1}{2^{\alpha +\beta +1}}\frac{\Gamma (n+\beta +1)\Gamma
(n+\alpha +\beta +1)}{\Gamma (n)\Gamma (\beta +1)\Gamma (\beta +2)\Gamma
(n+\alpha )}
\end{equation*}%
and 
\begin{equation*}
K_{n-1}(1,1)=\frac{1}{2^{\alpha +\beta +1}}\frac{\Gamma (n+\alpha +1)\Gamma
(n+\alpha +\beta +1)}{\Gamma (n)\Gamma (\alpha +1)\Gamma (\alpha +2)\Gamma
(n+\beta )}.
\end{equation*}

Recently, several authors (\cite{WKB96}, \cite{DimitrovMelloRafaeli}, \cite%
{DuenasMarcellan08}) have been contributed to the analysis of the behavior
of the zeros of $P_{n}^{\alpha,\beta,M}(x),$ and $P_{n}^{\alpha,\beta,N}(x)$.

Let denote by $(x_{n,k}^{M}(\alpha ))$, $(x_{n,k}^{N}(\alpha )),$ and $%
(x_{n,k}(\alpha ))$ the zeros of $P_{n}^{\alpha,\beta,M}(x)$, $%
P_{n}^{\alpha,\beta,N}(x),$ and $P_{n}^{\alpha ,\beta}(x)$, respectively,
all arranged in an increasing order. Then, applying the results of Section
3, we obtain

\setcounter{thm}{4}

\begin{thm}
\label{T.6J} The inequalities 
\begin{equation*}
\begin{array}{r}
-1<x_{n,1}^{M}(\alpha ,\beta )<x_{n,1}(\alpha,\beta)<x_{n-1,1}(\alpha ,\beta
+2)<x_{n,2}^{N}(\alpha ,\beta )<x_{n,2}(\alpha ,\beta )<\cdots \vspace{0.3cm}
\\ 
<x_{n-1,n-1}(\alpha ,\beta +2)<x_{n,n}^{N}(\alpha ,\beta )<x_{n,n}(\alpha
,\beta )%
\end{array}%
\end{equation*}%
hold for every $\alpha ,\beta >-1$. Moreover, each $x_{n,k}^{M}(\alpha
,\beta )$ is a decreasing function of $M$ and, for each $k=1,\ldots ,n-1$, 
\begin{equation}
\lim_{M\rightarrow \infty }x_{n,1}^{M}(\alpha ,\beta )=-1,\ \
\lim_{M\rightarrow \infty }x_{n,k+1}^{M}(\alpha ,\beta )=x_{n-1,k}(\alpha
,\beta +2),  \label{EqT5-2}
\end{equation}%
and 
\begin{equation*}
\begin{array}{l}
\lim\limits_{M\rightarrow \infty }M[x_{n,1}^{M}(\alpha ,\beta
)+1]=h_{n}(\alpha ,\beta ),\vspace{0.3cm} \\ 
\lim\limits_{M\rightarrow \infty }M[x_{n,k+1}^{M}(\alpha ,\beta
)-x_{n-1,k}(\alpha ,\beta +2)]=\displaystyle\frac{\left[ 1-x_{n-1,k}(\alpha
,\beta +2)\right] h_{n}(\alpha ,\beta )}{2(\beta +2)},%
\end{array}%
\end{equation*}%
where 
\begin{equation*}
h_{n}(\alpha ,\beta )=\displaystyle\frac{2^{\alpha +\beta +2}\Gamma
(n)\Gamma (\beta +2)\Gamma (\beta +3)\Gamma (n+\alpha )}{\Gamma (n+\beta
+2)\Gamma (n+\alpha +\beta +2)}.
\end{equation*}
\end{thm}

\begin{proof}
It remains to show the limits above. From (\ref{LimZeros-1}) 
\begin{equation*}
\lim\limits_{M\rightarrow \infty }M[x_{n,1}^{M}(\alpha ,\beta)+1] =\frac{%
-P_{n}^{\alpha,\beta}(-1)}{K_{n-1}(-1,-1)P_{n-1}^{\alpha,\beta+2}(-1)}
\end{equation*}
Since 
\begin{equation*}
P_{n}^{\alpha ,\beta }(-1)=\displaystyle\frac{(-1)^{n}2^{n}\Gamma (n+\beta
+1)\Gamma (n+\alpha +\beta +1)}{\Gamma (\beta +1)\Gamma (2n+\alpha +\beta +1)%
}.
\end{equation*}
and 
\begin{equation*}
K_{n-1}(-1,-1)=\frac{1}{2^{\alpha +\beta +1}}\frac{\Gamma (n+\beta +1)\Gamma
(n+\alpha +\beta +1)}{\Gamma (n)\Gamma (\beta +1)\Gamma (\beta
+2)\Gamma(n+\alpha )}
\end{equation*}
we obtain 
\begin{equation*}
\frac{-P_{n}^{\alpha,\beta}(-1)}{K_{n-1}(-1,-1)P_{n-1}^{\alpha,\beta+2}(-1)}%
= \displaystyle\frac{2^{\alpha +\beta +2}\Gamma (n)\Gamma (\beta +2)\Gamma
(\beta +3)\Gamma (n+\alpha )}{\Gamma (n+\beta +2)\Gamma (n+\alpha +\beta +2)}%
=h_{n}(\alpha ,\beta ).
\end{equation*}

Also from (\ref{LimZeros-1}) 
\begin{equation*}
\begin{array}{l}
\lim\limits_{M\rightarrow \infty }M[x_{n,k+1}^{M}(\alpha
,\beta)-x_{n-1,k}(\alpha ,\beta +2)]\vspace{0.3cm} \\ 
\hspace{0.5cm}\displaystyle=\frac{-P_{n}^{\alpha,\beta}(x_{n-1,k}(\alpha,%
\beta+2))}{K_{n-1}(-1,-1)(x_{n-1,k}(\alpha,\beta+2)+1)[P_{n-1}^{\alpha,%
\beta+2}(x)]^{\prime }\big|_{x=x_{n-1,k}(\alpha,\beta+2)}}.%
\end{array}%
\end{equation*}
On the other hand, it follows from 
\begin{equation*}
n(n+\alpha)(1+x)P_{n-1}^{\alpha,\beta+2}(x)=n(n+\alpha+\beta+1)P_{n}^{%
\alpha,\beta}(x)+(\beta+1)(1-x)[P_{n}^{\alpha,\beta}(x)]^{\prime }
\end{equation*}
that 
\begin{equation*}
\begin{array}{l}
n(n+\alpha+\beta+1)P_{n}^{\alpha,\beta}(x_{n-1,k}(\alpha,\beta+2))\vspace{%
0.3cm} \\ 
\hspace{0.5cm}=-(\beta+1)(1-x_{n-1,k}(\alpha,\beta+2))[P_{n}^{\alpha,%
\beta}(x)]^{\prime }\big|_{x=x_{n-1,k}(\alpha,\beta+2)}%
\end{array}%
\end{equation*}
and 
\begin{equation*}
\begin{array}{l}
n(n+\alpha)(1+x_{n-1,k}(\alpha,\beta+2))[P_{n}^{\alpha,\beta}(x)]^{\prime }%
\big|_{x=x_{n-1,k}(\alpha,\beta+2)}\vspace{0.3cm} \\ 
\hspace{0.5cm}=\displaystyle \lbrack
n(n+\alpha+\beta+1)-(\beta+1)][P_{n}^{\alpha,\beta}(x)]^{\prime }\big|%
_{x=x_{n-1,k}(\alpha,\beta+2)}\vspace{0.3cm} \\ 
\hspace{0.8cm}+\displaystyle(\beta+1)(1-x_{n-1,k}(\alpha,\beta+2))[P_{n}^{%
\alpha,\beta}(x)]^{\prime \prime }\big|_{x=x_{n-1,k}(\alpha,\beta+2)}.%
\end{array}%
\end{equation*}
Now, using the last two equalities and the differential equation for the
Jacobi polynomials 
\begin{equation*}
(1-x^2)[P_{n}^{\alpha,\beta}(x)]^{\prime \prime
}+[\beta-\alpha-(\alpha+\beta+1)x][P_{n}^{\alpha,\beta}(x)]^{\prime
}+n(n+\alpha+\beta+1)P_{n}^{\alpha,\beta}(x)=0
\end{equation*}
we obtain 
\begin{equation*}
\begin{array}{l}
(1+x_{n-1,k}(\alpha,\beta+2))[P_{n}^{\alpha,\beta+2}(x)]^{\prime }\big|%
_{x=x_{n-1,k}(\alpha,\beta+2)}\vspace{0.3cm} \\ 
\hspace{0.5cm}=\displaystyle\frac{-(n+\beta+1)(n+\alpha+\beta+1)}{%
(\beta+1)(1-x_{n-1,k}(\alpha,\beta+2))}P_{n}^{\alpha,\beta}(x_{n-1,k}(%
\alpha,\beta+2)).%
\end{array}%
\end{equation*}
Therefore 
\begin{equation*}
\begin{array}{l}
\lim\limits_{M\rightarrow \infty }M[x_{n,k+1}^{M}(\alpha
,\beta)-x_{n-1,k}(\alpha ,\beta +2)]\vspace{0.3cm} \\ 
\hspace{0.5cm}\displaystyle=\frac{-P_{n}^{\alpha,\beta}(x_{n-1,k}(\alpha,%
\beta+2))}{K_{n-1}(-1,-1)(x_{n-1,k}(\alpha,\beta+2)+1)[P_{n-1}^{\alpha,%
\beta+2}(x)]^{\prime }\big|_{x=x_{n-1,k}(\alpha,\beta+2)}}\vspace{0.3cm} \\ 
\hspace{0.5cm}\displaystyle=\frac{\left[ 1-x_{n-1,k}(\alpha,\beta +2)\right]
h_{n}(\alpha ,\beta )}{2(\beta +2)}.%
\end{array}%
\end{equation*}
\end{proof}

\setcounter{thm}{5}

\begin{thm}
\label{T.Jac} The inequalities 
\begin{equation*}
\begin{array}{l}
x_{n,1}(\alpha ,\beta )<x_{n,1}^{N}(\alpha ,\beta )<x_{n-1,1}(\alpha
+2,\beta )<\cdots <\vspace{0.3cm} \\ 
x_{n,n-1}(\alpha ,\beta )<x_{n,n-1}^{N}(\alpha ,\beta )<x_{n-1,n-1}(\alpha
+2,\beta )<x_{n,n}(\alpha ,\beta )<x_{n,n}^{N}(\alpha ,\beta )<1%
\end{array}%
\end{equation*}%
hold for every $\alpha ,\beta >-1$. Moreover, each $x_{n,k}^{N}(\alpha
,\beta )$ is an increasing function of $N$ and, for each $k=1,\ldots ,n-1$, 
\begin{equation*}
\lim_{N\rightarrow \infty }x_{n,n}^{N}(\alpha ,\beta )=1,\ \
\lim_{N\rightarrow \infty }x_{n,k}^{N}(\alpha ,\beta )=x_{n-1,k}(\alpha
+2,\beta ),
\end{equation*}%
and 
\begin{equation*}
\begin{array}{l}
\lim\limits_{N\rightarrow \infty }N[1-x_{n,n}^{N}(\alpha ,\beta
)]=g_{n}(\alpha ,\beta ),\vspace{0.3cm} \\ 
\lim\limits_{N\rightarrow \infty }N[x_{n-1,k}(\alpha +2,\beta
)-x_{n,k}^{N}(\alpha ,\beta )]=\displaystyle\frac{\left[ 1+x_{n-1,k}(\alpha
+2,\beta )\right] g_{n}(\alpha ,\beta )}{2(\alpha +2)},%
\end{array}%
\end{equation*}%
where 
\begin{equation*}
g_{n}(\alpha ,\beta )=\displaystyle\frac{2^{\alpha +\beta +2}\Gamma
(n)\Gamma (\alpha +2)\Gamma (\alpha +3)\Gamma (n+\beta )}{\Gamma (n+\alpha
+2)\Gamma (n+\alpha +\beta +2)}.
\end{equation*}
\end{thm}

\begin{proof}
The prove follows in the same way as the Theorem \ref{T.6J}. We only observe
that 
\begin{equation*}
\displaystyle\frac{n+\beta }{2n}(x-1)P_{n-1}^{\alpha +2,\beta
}(x)=P_{n}^{\alpha ,\beta }(x)-\frac{\alpha +1}{n(n+\alpha +\beta +1)}%
(1+x)[P_{n}^{\alpha ,\beta }(x)]^{\prime }.
\end{equation*}
\end{proof}

In order to illustrate the results of Theorem \ref{T.Jac}, we enclose the
graphs of $P_{3}^{\alpha ,\beta ,N+\varepsilon}(x)$, for $\alpha =\beta =0 $
and some values of $\varepsilon >0$, in order to show the monotonicity of
the zeros of $P_{3}^{\alpha,\beta,N}(x)$ as a function of the mass $N$ ( see
Figure \ref{FigJac}). 
\begin{figure}[h]
\centerline{\includegraphics[width=11cm,keepaspectratio]{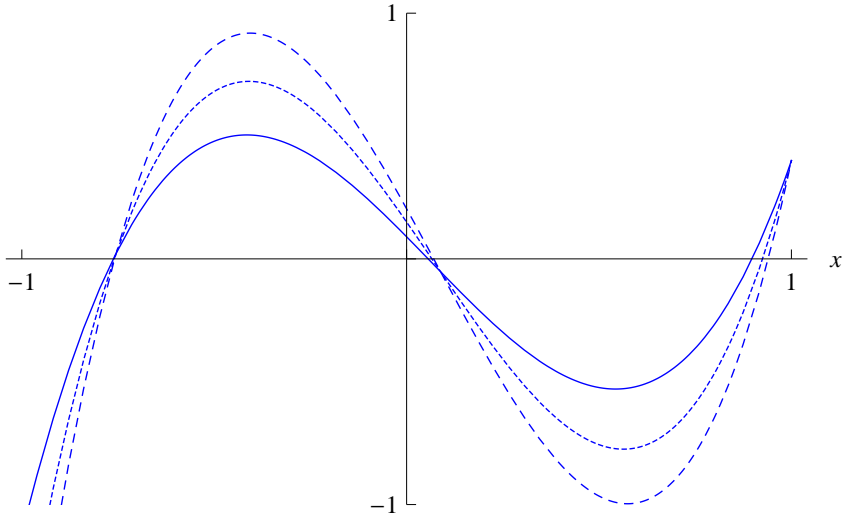}}
\caption{The graphs of $P_{3}^{\alpha ,\beta ,N+\varepsilon }(x)$ for some values of $\protect\varepsilon $.}
\label{FigJac}
\end{figure}

In Table \ref{tableJac} we show the value of zeros of $P_{3}^{\alpha ,\beta
,N}(x)$, with $\alpha =\beta =0$, for several choices of $N$ 
\begin{table}[th]
\caption{Zeros of $P_{3}^{\protect\alpha ,\protect\beta ,N}(x)$ for some
values of $N$.}
\label{tableJac}\centering
\begin{tabular}{>{\columncolor[gray]{.9}}p{0.2\linewidth}>{\columncolor[gray]{.9}}p{0.2\linewidth}>{\columncolor[gray]{.9}}p{0.2\linewidth}>{\columncolor[gray]{.9}}p{0.2\linewidth}}
\hline
$N$ & $x_{3,1}^{N}(0,0)$ & $x_{3,2}^{N}(0,0)$ & $x_{3,3}^{N}(0,0)$ \\ \hline
$0$ & $-0.774597$ & $0$ & $0.774597$ \\ 
$1$ & $-0.757872$ & $0.0753429$ & $0.955257$ \\ 
$10$ & $-0.755305$ & $0.0868168$ & $0.994575$ \\ 
$100$ & $-0.755004$ & $0.0881528$ & $0.999446$ \\ 
$1000$ & $-0.754974$ & $0.0882886$ & $0.999944$ \\ \hline
\end{tabular}%
\end{table}
Notice that the largest zero converges to $1$ and the other two zeros
converge to the zeros of the Jacobi polynomial $P_{2}^{2,0}(x)$, that is,
they converge to $x_{2,1}(2,0)=-0.75497$ and $x_{2,2}(4)=0.0883037$. Note
also that all the zeros increase when $N$ increase.

\subsection{Laguerre type (Laguerre-Koornwinder) orthogonal polynomials}

Next, we will deal with $p_{n}(x)=L_{n}^{\alpha }(x)$, that is, the
classical monic Laguerre polynomials, which are orthogonal with respect to
the measure $d\mu _{\alpha }=x^{\alpha }e^{-x}dx$, $\alpha >-1$, supported
on $[0,+{\infty })$. We will consider the Uvarov perturbation on $\mu
_{\alpha }$ with $a=0$ 
\begin{equation}
d\mu _{N}=d\mu _{\alpha }+N\delta _{0},N\geq 0.  \label{Laguerre-type}
\end{equation}%
The polynomials $L_{n}^{\alpha ,N}(x),$ orthogonal with respect to (\ref%
{Laguerre-type}), were also obtained by T. H. Koornwinder \cite%
{Koornwinder84} as a special limit case of the Jacobi-Koornwinder (Jacobi
type) orthogonal polynomials. In this direction, concerning analytic
properties of these polynomials, many contributions have been done in the
last years (see \cite{WKB96}, \cite{DimitrovMarcellanRafaeli}, \cite%
{DuenasMarcellan07}, \cite{Koekoek_1990}, among others). The connection
formula of $L_{n}^{\alpha ,N}(x)$ is 
\begin{equation}
L_{n}^{\alpha ,N}(x)=L_{n}^{\alpha }(x)+NK_{n-1}(0,0)\,x\,L_{n-1}^{\alpha
+2}(x),  \label{Lag-conexion}
\end{equation}%
where 
\begin{equation*}
K_{n-1}(0,0)=\frac{\Gamma (n+\alpha +1)}{\Gamma (n)\Gamma (\alpha +1)\Gamma
(\alpha +2)}.
\end{equation*}

Now, we will analyze the behavior of their zeros. Let denote by $%
(x_{n,k}^{N}(\alpha))$ and $(x_{n,k}(\alpha))$ the zeros of the Laguerre
type and the classical Laguerre orthogonal polynomials, respectively,
arranged in an increasing order. Applying the results of Section 4, we obtain

\setcounter{thm}{6}

\begin{thm}
\label{T.Lag} \label{TLag} The inequalities 
\begin{equation*}
\begin{array}{r}
0<x_{n,1}^{N}(\alpha )<x_{n,1}(\alpha )<x_{n-1,1}(\alpha
+2)<x_{n,2}^{N}(\alpha )<x_{n,2}(\alpha )<\cdots \vspace{0.3cm} \\ 
<x_{n-1,n-1}(\alpha +2)<x_{n,n}^{N}(\alpha )<x_{n,n}(\alpha )%
\end{array}%
\end{equation*}%
hold for every $\alpha >-1$. Moreover, each $x_{n,k}^{N}(\alpha )$ is a
decreasing function of $N$ and, for each $k=1,\ldots ,n-1$, 
\begin{equation*}
\lim_{N\rightarrow \infty }x_{n,1}^{N}(\alpha )=0,\ \ \lim_{N\rightarrow
\infty }x_{n,k+1}^{N}(\alpha )=x_{n-1,k}(\alpha +2),
\end{equation*}%
as well as 
\begin{equation*}
\begin{array}{l}
\lim\limits_{N\rightarrow \infty }Nx_{n,1}^{N}(\alpha )=g_{n}(\alpha ),%
\vspace{0.3cm} \\ 
\lim\limits_{N\rightarrow \infty }N[x_{n,k+1}^{N}(\alpha )-x_{n-1,k}(\alpha
+2)]=\displaystyle\frac{g_{n}(\alpha )}{\alpha +2},%
\end{array}%
\end{equation*}
where 
\begin{equation}  \label{gn}
g_{n}(\alpha)=\displaystyle\frac{\Gamma(n)\Gamma(\alpha+2)\Gamma(\alpha+3)}{%
\Gamma(n+\alpha+2)}.
\end{equation}
\end{thm}

\begin{proof}
It remains to show the limits above. From (\ref{LimZeros-1}) 
\begin{equation*}
\lim\limits_{N\rightarrow \infty }Nx_{n,1}^{N}(\alpha )=\frac{%
-L_{n}^{\alpha}(0)}{K_{n-1}(0,0)L_{n-1}^{\alpha+2}(0)}.
\end{equation*}
Since 
\begin{equation*}
L_{n}^{\alpha}(0)=\frac{(-1)^n\Gamma(n+\alpha+1)}{\Gamma(\alpha+1)}\ \ %
\mbox{and}\ \ K_{n-1}(0,0)=\frac{\Gamma(n+\alpha +1)}{\Gamma(n)\Gamma(%
\alpha+1)\Gamma(\alpha +2)},
\end{equation*}
we obtain 
\begin{equation*}
\displaystyle\frac{-L_{n}^{\alpha}(0)}{K_{n-1}(0,0)L_{n-1}^{\alpha+2}(0)}= %
\displaystyle\frac{\Gamma(n)\Gamma(\alpha+2)\Gamma(\alpha+3)}{%
\Gamma(n+\alpha+2)}=g_{n}(\alpha).
\end{equation*}

Also from (\ref{LimZeros-1}) 
\begin{equation*}
\begin{array}{l}
\displaystyle\lim\limits_{N\rightarrow \infty }N[x_{n,k+1}^{N}(\alpha
)-x_{n-1,k}(\alpha+2)]\vspace{0.3cm} \\ 
\hspace{0.5cm}\displaystyle=\frac{-L_{n}^{\alpha}(x_{n-1,k}(\alpha+2))}{%
K_{n-1}(0,0)x_{n-1,k}(\alpha+2)[L_{n-1}^{\alpha+2}(x)]^{\prime }\big|%
_{x=x_{n-1,k}(\alpha+2)}}.%
\end{array}%
\end{equation*}
On the other hand, it is easily to verify that 
\begin{equation*}
xL_{n-1}^{\alpha+2}(x)=L_{n}^{\alpha}(x)+\displaystyle\frac{\alpha+1}{n}%
[L_{n}^{\alpha}(x)]^{\prime }.
\end{equation*}
Thus, 
\begin{equation*}
[L_{n}^{\alpha}(x)]^{\prime }\big|_{x=x_{n-1,k}(\alpha+2)}=\displaystyle-%
\frac{n}{\alpha+1}L_{n}^{\alpha}(x_{n-1,k}(\alpha+2))
\end{equation*}
and 
\begin{equation*}
\begin{array}{l}
x_{n-1,k}(\alpha+2)[L_{n}^{\alpha}(x)]^{\prime }\big|_{x=x_{n-1,k}(\alpha+2)}%
\vspace{0.3cm} \\ 
\hspace{0.5cm}=\displaystyle\lbrack L_{n}^{\alpha}(x)]^{\prime }\big|%
_{x=x_{n-1,k}(\alpha+2)} +\displaystyle\frac{\alpha+1}{n}[L_{n}^{%
\alpha}(x)]^{\prime \prime }\big|_{x=x_{n-1,k}(\alpha+2)}.%
\end{array}%
\end{equation*}
Now, using the last two equalities and the differential equation for the
Laguerre polynomials 
\begin{equation*}
x[L_{n}^{\alpha}(x)]^{\prime \prime
}+(\alpha+1-x)[L_{n}^{\alpha}(x)]^{\prime }+nL_{n}^{\alpha}(x)=0
\end{equation*}
we obtain 
\begin{equation*}
x_{n-1,k}(\alpha+2)[L_{n}^{\alpha}(x)]^{\prime }\big|_{x=x_{n-1,k}(%
\alpha+2)}= \displaystyle\frac{-(n+\alpha+1)}{\alpha+1}L_{n}^{%
\alpha}(x_{n-1,k}(\alpha+2)).
\end{equation*}
Therefore 
\begin{equation*}
\begin{array}{l}
\displaystyle\lim\limits_{N\rightarrow \infty }N[x_{n,k+1}^{N}(\alpha
)-x_{n-1,k}(\alpha+2)]\vspace{0.3cm} \\ 
\hspace{0.5cm}\displaystyle=\frac{-L_{n}^{\alpha}(x_{n-1,k}(\alpha+2))}{%
K_{n-1}(0,0)x_{n-1,k}(\alpha+2)[L_{n-1}^{\alpha+2}(x)]^{\prime }\big|%
_{x=x_{n-1,k}(\alpha+2)}}\vspace{0.3cm} \\ 
\hspace{0.5cm}\displaystyle=\frac{\Gamma(n)\Gamma(\alpha+2)\Gamma(\alpha+2)}{%
\Gamma(n+\alpha+2)}\vspace{0.3cm} \\ 
\hspace{0.5cm}\displaystyle=\frac{g_{n}(\alpha)}{\alpha+2}.%
\end{array}%
\end{equation*}
\end{proof}

In order to illustrate the results of Theorem \ref{T.Lag}, we enclose the
graphs of $L_{3}^{\alpha ,N+\varepsilon }(x)$, for $\alpha =2$ and some
values of $\varepsilon >0$, in order to show the monotonicity of the zeros
of $L_{3}^{\alpha ,N}(x)$ as a function of the mass $N$. See Figure \ref%
{FigLag}. 
\begin{figure}[th]
\centerline{\includegraphics[width=11cm,keepaspectratio]{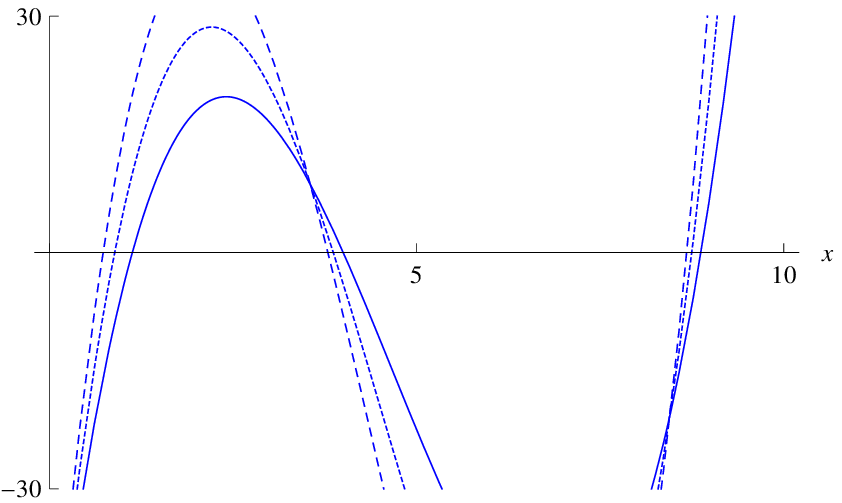}}
\caption{The graphs of $L_{3}^{\alpha ,N+\varepsilon}(x)$
for some values of $\protect\varepsilon $.}
\label{FigLag}
\end{figure}

Table \ref{tableLag} shows the zeros of $L_{3}^{\alpha ,N}(x)$, with $\alpha
=2$, for several choices of $N$. 
\begin{table}[th]
\caption{Zeros of $L_{3}^{\protect\alpha ,N}(x)$ for some values of $N$.}
\label{tableLag}\centering
\begin{tabular}{>{\columncolor[gray]{.9}}p{0.2\linewidth}>{\columncolor[gray]{.9}}p{0.2\linewidth}>{\columncolor[gray]{.9}}p{0.2\linewidth}>{\columncolor[gray]{.9}}p{0.2\linewidth}}
\hline
$N$ & $x_{3,1}^{N}(2)$ & $x_{3,2}^{N}(2)$ & $x_{3,3}^{N}(2)$ \\ \hline
$0$ & $1.51739$ & $4.31158$ & $9.17103$ \\ 
$1$ & $0.321731$ & $3.64053$ & $8.53774$ \\ 
$10$ & $0.0390611$ & $3.5604$ & $8.45936$ \\ 
$100$ & $0.00399042$ & $3.55151$ & $8.45049$ \\ 
$1000$ & $0.00039990$ & $3.55061$ & $8.44959$ \\ \hline
\end{tabular}%
\end{table}
Observe that the smallest zero converges to $0$ and the other two zeros
converge to the zeros of the Laguerre polynomial $L_{2}^{4}(x)$, that is,
they converge to $x_{2,1}(4)=3.55051$ and $x_{2,2}(4)=8.44949$. Notice that
all the zeros decrease when $N$ increases.

\subsection{Hermite type orthogonal polynomials}

Using the symmetrization process for the special case of Laguerre type
orthogonal polynomials when $\alpha = -1/2$, we obtain the Hermite type
orthogonal polynomials $H_{n}^{N}(x)$, which are orthogonal with respect to
the measure 
\begin{equation*}
d\mu _{N}=e^{-x^{2}}dx+N\delta _{0}.
\end{equation*}%
whose support is the real line. It easy to see that they are symmetric with
respect to the origin and 
\begin{equation}
H_{2n}^{N}(x)=L_{n}^{(-1/2,N)}(x^{2})  \label{H2n}
\end{equation}%
and 
\begin{equation}
H_{2n+1}^{N}(x)=H_{2n+1}^{0}(x)=xL_{n}^{(1/2,0)}(x^{2}).  \label{H2n+1}
\end{equation}%
Let denote by $(h_{2n,k}^{N})$, $1\leq k\leq 2n$, and $(h_{2n+1,k})$, $1\leq
k\leq 2n+1$, the zeros of the Hermite type polynomials $H_{2n}^{N}(x)$ and $%
H_{2n+1}^{N}(x)$, respectively, ordered as follows : $h_{2n,n}^{N}<\cdots
<h_{2n,1}^{N}(\alpha )$ and $h_{2n+1,n}<\cdots <h_{2n+1,1}$. Because of the
symmetry property of these polynomials, we get $%
h_{2n,k}^{N}=-h_{2n,2n-k+1}^{N}$ and $h_{2n+1,k}=-h_{2n+1,2n+2-k}$, $%
k=1,\ldots ,n$. Furthermore, from (\ref{H2n}) and (\ref{H2n+1}), we have 
\begin{equation*}
\left[ h_{2n,k}^{N}\right] ^{2}=x_{n,n-k+1}^{N}(-1/2)
\end{equation*}%
and 
\begin{equation*}
\left[ h_{2n+1,k}\right] ^{2}=x_{n,n-k+1}(1/2).
\end{equation*}%
Then, as a straightforward consequence of Theorem \ref{TLag}, we get

\setcounter{thm}{7}

\begin{thm}
Let $n\in\mathbb{N}$. Then\newline

\begin{itemize}
\item[(i)] The inequalities 
\begin{equation*}
\begin{array}{r}
0<\left[ h_{2n,n}^{N}\right] ^{2}<x_{n,1}(-1/2)<x_{n-1,1}(3/2)<\left[
h_{2n,n-1}^{N}\right] ^{2}<x_{n,2}(-1/2)<\cdots \vspace{0.3cm} \\ 
<x_{n-1,n-1}(3/2)<\left[ h_{2n,1}^{N}\right] ^{2}<x_{n,n}(-1/2)%
\end{array}%
\end{equation*}%
hold. Moreover, each $h_{2n,k}^{N}$, $k=1,\ldots ,n$, is a decreasing
function of $N$ and, for each $k=1,\ldots ,n-1$, 
\begin{equation*}
\lim_{N\rightarrow \infty }\left[ h_{2n,n}^{N}\right] ^{2}=0,\ \
\lim_{N\rightarrow \infty }\left[ h_{2n,k}^{N}\right] ^{2}=x_{n-1,n-k}(3/2),
\end{equation*}%
and 
\begin{equation}
\begin{array}{l}
\lim\limits_{N\rightarrow \infty }N\left[ h_{2n,n}^{N}\right]
^{2}=g_{n}(-1/2),\vspace{0.3cm} \\ 
\lim\limits_{N\rightarrow \infty }N[\left[ h_{2n,k}^{N}\right]
^{2}-x_{n-1,n-k}(3/2)]=\displaystyle\frac{g_{n}(-1/2)}{3/2};%
\end{array}
\label{ggn}
\end{equation}

\item[(ii)] The inequalities 
\begin{equation*}
\begin{array}{r}
0<\left[h_{2n+1,n}\right]^2<x_{n,1}(1/2)<x_{n-1,1}(5/2)<\left[h_{2n+1,n-1}%
\right]^2<x_{n,2}(1/2)<\cdots\vspace{0.3cm} \\ 
<x_{n-1,n-1}(5/2)<\left[h_{2n,1}\right]^2<x_{n,n}(1/2)%
\end{array}%
\end{equation*}
hold.
\end{itemize}
\end{thm}

\noindent The function $g_{n}(\alpha )$ in (\ref{ggn}) is defined in (\ref%
{gn}). More details about the zeros of Hermite-type polynomials can be found
in \cite{MarcRafa2010}.


\section{Electrostatic interpretation}

\subsection{Some preliminaries results}

We also provided an alternative connection formula for the MOPS $%
\{p_{n}^{N}(a;x)\}_{n\geq0}$ in terms of the kernel polynomials $\{
p_{n}^{\ast}(a;x)\}_{n\geq0}$.

\setcounter{thm}{8}

\begin{thm}
The polynomials $\{p_{n}^{N}(a;x)\}_{n\geq0}$ can be also represented as 
\begin{equation}  \label{02nd-Rel}
p_{n}^{N}(a;x)=p_{n}^{\ast}(a;x)+c_{n}\,p_{n-1}^{\ast}(a;x),
\end{equation}
where%
\begin{equation}  \label{cn}
c_{n}=- \frac{1+NK_{n}(a,a)}{1+NK_{n-1}(a,a)}\frac{p_{n-1}(a)}{p_{n}(a)}%
\,\gamma _{n}\ \ \ \mbox{and}\ \ \ \gamma _{n}=\dfrac{\Vert p_{n}\Vert _{\mu
}^{2}}{\Vert p_{n-1}\Vert _{\mu}^{2}}.
\end{equation}
\end{thm}

\begin{proof}
Using (\ref{Kernel}), we can write $p_{n}(x)$ as 
\begin{equation}  \label{p11}
p_{n}(x)=\frac{\Vert p_{n}\Vert _{\mu} ^{2}}{p_{n}(a)}\left[
K_{n}(a,x)-K_{n-1}(a,x)\right].
\end{equation}
Now, from (\ref{p1}) and (\ref{p11}), we have 
\begin{equation}  \label{pneqKnKn-1}
p_{n}(x)=p_{n}^{\ast}(a;x)-\gamma _{n}\frac{p_{n-1}(a)}{p_{n}(a)}%
p_{n-1}^{\ast}(a;x).
\end{equation}%
Therefore, substituting (\ref{p1}) and (\ref{pneqKnKn-1}) in (\ref{Con-For}%
), we obtain 
\begin{equation*}
\begin{array}{lll}
p_{n}^{N}(a;x) & = & \displaystyle\left(p_{n}^{\ast}(a;x)-\gamma _{n}\frac{%
p_{n-1}(a)}{p_{n}(a)}p_{n-1}^{\ast}(a;x)\right)\vspace{0.3cm} \\ 
&  & \displaystyle\hspace{1.5cm}-\frac{Np_{n}(a)}{1+NK_{n-1}(a,a)}\dfrac{%
p_{n-1}(a)}{\Vert p_{n-1}\Vert _{\mu}^{2}}p_{n-1}^{\ast}(a;x)\vspace{0.3cm}
\\ 
& = & p_{n}^{\ast}(a;x)+c_{n}\,p_{n-1}^{\ast}(a;x),%
\end{array}%
\end{equation*}
where $c_n$ is given in (\ref{cn}).
\end{proof}

\setcounter{thm}{9}

\begin{thm}
The MOPS $\{p_{n}^{\ast }(a;x)\}_{n\geq 0}$ satisfies the three-term
recurrence relation 
\begin{equation}  \label{TTRRK}
x\,p_{n}^{\ast }(a;x)=p_{n+1}^{\ast }(a;x)+\beta _{n}^{\ast }\,p_{n}^{\ast
}(a;x)+\gamma _{n}^{\ast }\,p_{n-1}^{\ast }(a;x),\ \ n\geq 0,
\end{equation}%
with initial conditions $p_{0}^{\ast }(a;x)=1$ and $p_{-1}^{\ast }(a;x)=0$.
The coefficients of the TTRR are 
\begin{equation}
\beta _{n}^{\ast }=\beta _{n+1}+\dfrac{p_{n+2}(a)}{p_{n+1}(a)}-\dfrac{%
p_{n+1}(a)}{p_{n}(a)},\ \ \ n\geq 0,  \label{coef-RRTTK}
\end{equation}%
and 
\begin{equation}
\gamma _{n}^{\ast }=\dfrac{p_{n+1}(a)p_{n-1}(a)}{\left[ p_{n}(a)\right] ^{2}}%
\gamma _{n}>0\, \ \ n\geq 1.  \label{coef-RRTTK-2}
\end{equation}
\end{thm}

\begin{proof}
From (\ref{p1}) and (\ref{pneqKnKn-1}), we have%
\begin{eqnarray*}
(x-a)p_{n}^{\ast }(a;x) &=&p_{n+1}(x)-\frac{p_{n+1}(a)}{p_{n}(a)}p_{n}(x) \\
&=&p_{n+1}^{\ast }(a;x)-\gamma _{n+1}\frac{p_{n}(a)}{p_{n+1}(a)}p_{n}^{\ast
}(a;x) \\
&&\hspace{1.6cm}-\frac{p_{n+1}(a)}{p_{n}(a)}\left( p_{n}^{\ast }(a;x)-\gamma
_{n}\frac{p_{n-1}(a)}{p_{n}(a)}p_{n-1}^{\ast }(a;x)\right) \\
&=&p_{n+1}^{\ast }(a;x)-\left( \gamma _{n+1}\frac{p_{n}(a)}{p_{n+1}(a)}+%
\frac{p_{n+1}(a)}{p_{n}(a)}\right) p_{n}^{\ast }(a;x) \\
&&\hspace{4.3cm}+\gamma _{n}\frac{p_{n+1}(a)p_{n-1}(a)}{\left[ p_{n}(a)%
\right] ^{2}}p_{n-1}^{\ast }(a;x),
\end{eqnarray*}%
or, equivalently,%
\begin{eqnarray*}
xp_{n}^{\ast }(a;x) &=&p_{n+1}^{\ast }(a;x)-\left( \gamma _{n+1}\frac{%
p_{n}(a)}{p_{n+1}(a)}+\frac{p_{n+1}(a)}{p_{n}(a)}-a\right) p_{n}^{\ast }(a;x)
\\
&&\hspace{4.3cm}+\gamma _{n}\frac{p_{n+1}(a)p_{n-1}(a)}{\left[ p_{n}(a)%
\right] ^{2}}p_{n-1}^{\ast }(a;x).
\end{eqnarray*}%
Since%
\begin{equation*}
ap_{n+1}(a)=p_{n+2}(a)+\beta _{n+1}p_{n+1}(a)+\gamma _{n+1}p_{n}(a),
\end{equation*}%
we obtain 
\begin{equation*}
\gamma _{n+1}p_{n}(a)=ap_{n+1}(a)-p_{n+2}(a)-\beta _{n+1}p_{n+1}(a).
\end{equation*}%
Thus,%
\begin{equation*}
\begin{array}{lll}
\beta _{n}^{\ast } & = & \displaystyle-\left( \gamma _{n+1}\frac{p_{n}(a)}{%
p_{n+1}(a)}+\frac{p_{n+1}(a)}{p_{n}(a)}\right) +a\vspace{0.3cm} \\ 
& = & \displaystyle-\left( \frac{ap_{n+1}(a)-p_{n+2}(a)-\beta
_{n+1}p_{n+1}(a)}{p_{n+1}(a)}+\frac{p_{n+1}(a)}{p_{n}(a)}\right) +a\vspace{%
0.3cm} \\ 
& = & \displaystyle-\frac{ap_{n+1}(a)}{p_{n+1}(a)}+\frac{p_{n+2}(a)}{%
p_{n+1}(a)}+\frac{\beta _{n+1}p_{n+1}(a)}{p_{n+1}(a)}-\frac{p_{n+1}(a)}{%
p_{n}(a)}+a\vspace{0.3cm} \\ 
& = & \displaystyle\beta _{n+1}+\frac{p_{n+2}(a)}{p_{n+1}(a)}-\frac{%
p_{n+1}(a)}{p_{n}(a)}.%
\end{array}%
\end{equation*}

For the other coefficient, we have%
\begin{equation}
\gamma _{n}^{\ast }=\dfrac{\Vert p_{n}^{\ast }\Vert _{\mu ^{\ast }}^{2}}{%
\Vert p_{n-1}^{\ast }\Vert _{\mu ^{\ast }}^{2}}=\frac{\left\langle \left(
x-a\right) p_{n}^{\ast },p_{n}^{\ast }\right\rangle _{\mu }}{\left\langle
\left( x-a\right) p_{n-1}^{\ast },p_{n-1}^{\ast }\right\rangle _{\mu }}.
\label{gammastar1}
\end{equation}

But according to (\ref{p1}) 
\begin{equation*}
\left( x-a\right) p_{n}^{\ast }(a;x)=p_{n+1}(x)-\dfrac{p_{n+1}(a)}{p_{n}(a)}%
p_{n}(x),
\end{equation*}

and%
\begin{equation*}
\left\langle \left( x-a\right) p_{n}^{\ast }\left( a;x\right) ,p_{n}^{\ast
}\left( a;x\right) \right\rangle _{\mu }=-\dfrac{p_{n+1}(a)}{p_{n}(a)}%
\left\Vert p_{n}^{\ast }\right\Vert _{\mu }^{2}.
\end{equation*}%
Thus (\ref{gammastar1}) becomes%
\begin{equation*}
\gamma _{n}^{\ast }=\dfrac{p_{n+1}(a)p_{n-1}(a)}{\left[ p_{n}(a)\right] ^{2}}%
\gamma _{n}.
\end{equation*}
\end{proof}

\bigskip

As an example, we can analyze the behavior of these coefficients in the
Laguerre case. Let $\widehat{L}_{n}^{\alpha }\left( x\right) =\frac{\left(
-1\right) ^{n}}{n!}L_{n}^{\alpha }\left( x\right) $, then we have%
\begin{eqnarray*}
\frac{\beta _{n}^{\ast }}{\beta _{n+1}} &=&1-\frac{1}{\beta _{n+1}}\left(
n+2\right) \frac{\widehat{L}_{n+2}^{\alpha }\left( a\right) }{\widehat{L}%
_{n+1}^{\alpha }\left( a\right) }+\frac{1}{\beta _{n+1}}\left( n+1\right) 
\frac{\widehat{L}_{n+1}^{\alpha }\left( a\right) }{\widehat{L}_{n}^{\alpha
}\left( a\right) } \\
&=&1-\frac{n+2}{2n+\alpha +3}\left( 1+\frac{\sqrt{\left\vert a\right\vert }}{%
\sqrt{n+2}}+\mathcal{O}\left( \frac{1}{n+2}\right) \right) \\
&&+\frac{n+1}{2n+\alpha +3}\left( 1+\frac{\sqrt{\left\vert a\right\vert }}{%
\sqrt{n+1}}+\mathcal{O}\left( \frac{1}{n+1}\right) \right) \\
&=&1-\frac{1}{2n+\alpha +3}+\frac{\sqrt{\left\vert a\right\vert }}{2n+\alpha
+3}\left( -\sqrt{n+2}+\sqrt{n+1}\right) +\mathcal{O}\left( \frac{1}{n^{2}}%
\right) \\
&=&1-\frac{1}{2n+\alpha +3}-\frac{\sqrt{\left\vert a\right\vert }}{2n+\alpha
+3}\frac{1}{\sqrt{n+2}+\sqrt{n+1}}+\mathcal{O}\left( \frac{1}{n^{2}}\right)
\\
&=&1-\frac{1}{2n+\alpha +3}-\frac{\sqrt{\left\vert a\right\vert }}{2}%
n^{-3/2}+\mathcal{O}\left( \frac{1}{n^{2}}\right) .
\end{eqnarray*}

Thus%
\begin{eqnarray*}
\frac{\beta _{n}^{\ast }}{\beta _{n}} &=&\frac{2n+\alpha +3}{2n+\alpha +1}%
\left( 1-\frac{1}{2n+\alpha +3}-\frac{\sqrt{\left\vert a\right\vert }}{2}%
n^{-3/2}+\mathcal{O}\left( \frac{1}{n^{2}}\right) \right) \\
&=&\left( 1+\frac{2}{2n+\alpha +1}\right) \left( 1-\frac{1}{2n+\alpha +3}-%
\frac{\sqrt{\left\vert a\right\vert }}{2}n^{-3/2}+\mathcal{O}\left( \frac{1}{%
n^{2}}\right) \right) \\
&=&1+\frac{1}{2n}+\mathcal{O}\left( n^{-3/2}\right) .
\end{eqnarray*}%
\bigskip

On the other hand, taking into account%
\begin{equation*}
\gamma _{n}^{\ast }=\dfrac{p_{n+1}(a)p_{n-1}(a)}{\left[ p_{n}(a)\right] ^{2}}%
\gamma _{n},
\end{equation*}%
then%
\begin{eqnarray*}
\frac{\gamma _{n}^{\ast }}{\gamma _{n}} &=&\frac{\left( n+1\right) }{n}\frac{%
\widehat{L}_{n+1}^{\alpha }\left( a\right) }{\widehat{L}_{n}^{\alpha }\left(
a\right) }\frac{\widehat{L}_{n-1}^{\alpha }\left( a\right) }{\widehat{L}%
_{n}^{\alpha }\left( a\right) } \\
&=&\left( 1+\frac{1}{n}\right) \frac{1+\frac{\sqrt{\left\vert a\right\vert }%
}{\sqrt{n+1}}+\mathcal{O}\left( \frac{1}{n+1}\right) }{1+\frac{\sqrt{%
\left\vert a\right\vert }}{\sqrt{n}}+\mathcal{O}\left( \frac{1}{n}\right) }
\\
&=&\left( 1+\frac{1}{n}\right) \left( 1+\frac{\sqrt{\left\vert a\right\vert }%
}{\sqrt{n+1}}+\mathcal{O}\left( \frac{1}{n+1}\right) \right) \left( 1-\frac{%
\sqrt{\left\vert a\right\vert }}{\sqrt{n}}+\mathcal{O}\left( \frac{1}{n}%
\right) \right) \\
&=&\left( 1+\frac{1}{n}\right) \left( 1-\frac{\sqrt{\left\vert a\right\vert }%
}{2n^{3/2}}+\mathcal{O}\left( \frac{1}{n^{2}}\right) \right) \\
&=&1+\frac{1}{n}+\mathcal{O}\left( n^{-3/2}\right) .
\end{eqnarray*}

\bigskip

Next, we will assume that $d\mu ^{\ast }(x)=\left( x-a\right) \omega \left(
x\right) dx$ where $\omega (x)$ is a weight function supported on the real
line. We can associate with $\omega \left( x\right) $ an external potential $%
\upsilon \left( x\right) $ such that $\omega \left( x\right) =\exp \left(
-\upsilon \left( x\right) \right) $.

Notice that if $\upsilon(x)$ is assumed to be differentiable in the support
of $d\mu(x)=\omega \left( x\right) dx$ then 
\begin{equation*}
\frac{\omega ^{\prime }\left( x\right) }{\omega \left( x\right) }=-\upsilon
^{\prime }\left( x\right) .
\end{equation*}

If $\upsilon ^{\prime }(x)$ is a rational function, then the weight function 
$\omega (x)$ is said to be semi-classical (see \cite{Mar91}, \cite{Sho39}).
The linear functional $u$ associated with $\omega (x)$, i.e., 
\begin{equation*}
\left\langle u,p(x)\right\rangle =\int\limits_{\Sigma }p(x)\omega \left(
x\right) dx,
\end{equation*}%
satisfies a distributional equation (which is known in the literature as
Pearson equation)%
\begin{equation*}
D(\sigma (x)u)=\tau (x)u,
\end{equation*}%
where $\sigma (x)$ and $\tau (x)$ are polynomials such that $\sigma (x)$ is
monic and deg$(\tau (x))\geq 1$.

Notice that, in terms of the weight function, the above relation means that%
\begin{equation*}
\frac{\omega ^{\prime }(x)}{\omega (x)}=\frac{\tau (x)-\sigma ^{\prime }(x)}{%
\sigma (x)},
\end{equation*}%
or, equivalently, 
\begin{equation*}
\upsilon ^{\prime }(x)=-\frac{\tau (x)-\sigma ^{\prime }(x)}{\sigma (x)}.
\end{equation*}

Let consider the linear functional $u^{\ast }$ associated with the measure $%
\mu ^{\ast }(x)$. In order to find the Pearson equation that $u^{\ast }$
satisfies we will analyze two situations.

\begin{itemize}
\item[($i$)] If $\sigma (a)\neq 0,$ then 
\begin{equation*}
\begin{array}{lll}
D\left( (x-a)\sigma (x)u^{\ast }\right) & = & D\left(
(x-a)^{2}\sigma(x)u\right) \vspace{0.2cm} \\ 
& = & 2(x-a)\sigma (x)u+(x-a)^{2}D(\sigma (x)u) \vspace{0.2cm} \\ 
& = & 2\sigma (x)u^{\ast }+(x-a)^{2}\tau (x)u \vspace{0.2cm} \\ 
& = & \left[ 2\sigma (x)+(x-a)\tau (x)\right] u^{\ast }.%
\end{array}%
\end{equation*}
Thus, 
\begin{eqnarray*}
D\left(\phi(x)u^{\ast}\right)=\psi(x)u^{\ast},
\end{eqnarray*}
where 
\begin{equation}  \label{Pea1}
\left|%
\begin{array}{l}
\phi (x) =(x-a)\sigma (x) \vspace{0.2cm} \\ 
\psi (x) =2\sigma (x)+(x-a)\tau (x).%
\end{array}%
\right.
\end{equation}

\item[($ii$)] If $\sigma(a)=0$, i.e., $\sigma (x)=(x-a)\tilde{\sigma}(x),$
then 
\begin{equation*}
\begin{array}{lll}
D\left( \sigma (x)u^{\ast }\right) & = & D\left( (x-a)\tilde{\sigma}%
(x)u^{\ast}\right) \vspace{0.2cm} \\ 
& = & D\left( (x-a)^{2}\tilde{\sigma}(x)u\right) =D\left(
(x-a)\sigma(x)u\right) \vspace{0.2cm} \\ 
& = & \sigma (x)u+(x-a)D\left( \sigma (x)u\right) =\sigma (x)u+(x-a)\tau
(x)u \vspace{0.2cm} \\ 
& = & \left( \tilde{\sigma}(x)+\tau (x)\right) u^{\ast }.%
\end{array}%
\end{equation*}
In this case, 
\begin{eqnarray*}
D\left(\phi(x)u^{\ast}\right)=\psi(x)u^{\ast},
\end{eqnarray*}
with 
\begin{equation}  \label{Pea2}
\left|%
\begin{array}{l}
\phi (x) =\sigma (x) \vspace{0.2cm} \\ 
\psi (x) =\tilde{\sigma}(x)+\tau (x).%
\end{array}%
\right.
\end{equation}
\end{itemize}

It is a very well known result that the sequence of monic polynomials $%
\{p_{n}^{\ast }(a;x)\}_{n\geq 0},$ orthogonal with respect to $u^{\ast }$%
\thinspace\ satisfies a structure relation (see {\cite{Dini Maroni90}} and 
\cite{Mar91})%
\begin{equation}
\phi (x)D\left( p_{n}^{\ast }(a;x)\right) =A(x,n)p_{n}^{\ast
}(a;x)+B(x,n)p_{n-1}^{\ast }(a;x)  \label{Struct-Rel}
\end{equation}%
where $A(x,n)$ and $B(x,n)$ are polynomials of fixed degree, that do not
depend on $n$.

\setcounter{lem}{1}

\begin{lem}
\cite{Ism00-A}\ We have 
\begin{equation}
\displaystyle A(x,n)+A(x,n-1)+\frac{(x-\beta _{n-1}^{\ast })}{\gamma
_{n-1}^{\ast }}B(x,n-1)=\phi ^{\prime }(x)-\psi (x).  \label{lema22}
\end{equation}
\end{lem}

\begin{proof}
According to a result by Ismail (\cite{Ism00-A}, (1.12)) which must be
adapted to our situation since we use monic polynomials, we get%
\begin{equation*}
\begin{array}{lll}
\displaystyle A(x,n)+A(x,n-1)+\frac{(x-\beta _{n-1}^{\ast })}{\gamma
_{n-1}^{\ast }}B(x,n-1) & = & \displaystyle-\phi (x)\frac{\left[ \omega
^{\ast }(x)\right] ^{\prime }}{\omega ^{\ast }(x)}\vspace{0.2cm} \\ 
& = & \displaystyle-\phi (x)\frac{\psi (x)-\phi ^{\prime }(x)}{\phi (x)}%
\vspace{0.2cm} \\ 
& = & \displaystyle\phi ^{\prime }(x)-\psi (x),%
\end{array}%
\end{equation*}%
where $\omega ^{\ast }(x)=(x-a)\omega (x)$.
\end{proof}

Now, applying the derivative operator in (\ref{02nd-Rel}) and multiplying it
by $\phi (x)$, we obtain 
\begin{equation}
\phi (x)D\left( p_{n}^{N}(a;x)\right) =\phi (x)D\left( p_{n}^{\ast
}(a;x)\right) +c_{n}\phi (x)D\left( p_{n-1}^{\ast }(a;x)\right) .
\label{Struct-Rel2}
\end{equation}%
Thus, substituting (\ref{Struct-Rel}) in (\ref{Struct-Rel2}), yields 
\begin{equation}
\begin{array}{l}
\phi (x)D\left( p_{n}^{N}(a;x)\right) =A(x,n)p_{n}^{\ast }(a;x)+\left[
B(x,n)+c_{n}A(x,n-1)\right] p_{n-1}^{\ast }(a;x)\vspace{0.3cm} \\ 
\hspace{6.8cm}+c_{n}B(x,n-1)p_{n-2}^{\ast }(a;x).%
\end{array}
\label{Struct-Rel3}
\end{equation}%
Finally, using the TTRR (\ref{TTRRK}) in (\ref{Struct-Rel3}), we obtain 
\begin{equation}
\phi (x)\left( p_{n}^{N}(a;x)\right) ^{\prime }=A^{\ast }(x,n)p_{n}^{\ast
}(a;x)+B^{\ast }(x,n)p_{n-1}^{\ast }(a;x),  \label{3rd-Rel}
\end{equation}%
where 
\begin{equation}
A^{\ast }(x,n)=\left( A(x,n)-\dfrac{c_{n}}{\gamma _{n-1}^{\ast }}%
B(x,n-1)\right)  \label{Coeff_4}
\end{equation}%
and 
\begin{equation}
B^{\ast }(x,n)=\left( B(x,n)+c_{n}A(x,n-1)+\dfrac{c_{n}}{\gamma _{n-1}^{\ast
}}(x-\beta _{n-1}^{\ast })B(x,n-1)\right) .  \label{Coeff_5}
\end{equation}%
Therefore, from (\ref{02nd-Rel}) and (\ref{3rd-Rel}), it follows that 
\begin{equation*}
\left[ 
\begin{array}{cc}
1 & c_{n}\vspace{0.3cm} \\ 
A^{\ast }(x,n) & B^{\ast }(x,n)%
\end{array}%
\right] \left[ 
\begin{array}{c}
p_{n}^{\ast }(a;x)\vspace{0.3cm} \\ 
p_{n-1}^{\ast }(a;x)%
\end{array}%
\right] =\left[ 
\begin{array}{c}
p_{n}^{N}(a;x)\vspace{0.3cm} \\ 
\phi (x)D\left( p_{n}^{N}(a;x)\right)%
\end{array}%
\right] ,
\end{equation*}%
that is, 
\begin{equation*}
p_{n}^{\ast }(a;x)=\frac{%
\begin{vmatrix}
p_{n}^{N}(a;x) & c_{n}\vspace{0.3cm} \\ 
\phi (x)D\left( p_{n}^{N}(a;x)\right) & B^{\ast }(x,n)%
\end{vmatrix}%
}{%
\begin{vmatrix}
1 & c_{n}\vspace{0.3cm} \\ 
A^{\ast }(x,n) & B^{\ast }(x,n)%
\end{vmatrix}%
}
\end{equation*}%
and 
\begin{equation*}
p_{n-1}^{\ast }(a;x)=\frac{%
\begin{vmatrix}
1 & p_{n}^{N}(a;x)\vspace{0.3cm} \\ 
A^{\ast }(x,n) & \phi (x)D\left( p_{n}^{N}(a;x)\right)%
\end{vmatrix}%
}{%
\begin{vmatrix}
1 & c_{n}\vspace{0.3cm} \\ 
A^{\ast }(x,n) & B^{\ast }(x,n)%
\end{vmatrix}%
},
\end{equation*}%
or, equivalently, 
\begin{equation}
p_{n}^{\ast }(a;x)=\displaystyle\frac{B^{\ast }(x,n)}{B^{\ast
}(x,n)-c_{n}A^{\ast }(x,n)}\,p_{n}^{N}(a;x)-\frac{c_{n}\phi (x)}{B^{\ast
}(x,n)-c_{n}A^{\ast }(x,n)}\,D\left( p_{n}^{N}(a;x)\right)  \label{MonKer-n}
\end{equation}%
and 
\begin{equation}
p_{n-1}^{\ast }(a;x)=\displaystyle\frac{-A^{\ast }(x,n)}{B^{\ast
}(x,n)-c_{n}A^{\ast }(x,n)}p_{n}^{N}(a;x)+\frac{\phi (x)}{B^{\ast
}(x,n)-c_{n}A^{\ast }(x,n)}D\left( p_{n}^{N}(a;x)\right) .
\label{MonKer-n-1}
\end{equation}

Now, substituting (\ref{MonKer-n}) and (\ref{MonKer-n-1}) in (\ref%
{Struct-Rel}), we deduce 
\begin{equation*}
\begin{array}{l}
\phi (x)D\left( \displaystyle\frac{B^{\ast }(x,n)}{B^{\ast
}(x,n)-c_{n}A^{\ast }(x,n)}\,p_{n}^{N}(a;x)-\frac{c_{n}\phi (x)}{B^{\ast
}(x,n)-c_{n}A^{\ast }(x,n)}\,D\left( p_{n}^{N}(a;x)\right) \right) \vspace{%
0.3cm} \\ 
=A(x,n)\left( \displaystyle\frac{B^{\ast }(x,n)}{B^{\ast }(x,n)-c_{n}A^{\ast
}(x,n)}\,p_{n}^{N}(a;x)-\frac{c_{n}\phi (x)}{B^{\ast }(x,n)-c_{n}A^{\ast
}(x,n)}\,D\left( p_{n}^{N}(a;x)\right) \right) \vspace{0.3cm} \\ 
+B(x,n)\left( \displaystyle\frac{-A^{\ast }(x,n)}{B^{\ast
}(x,n)-c_{n}A^{\ast }(x,n)}p_{n}^{N}(a;x)+\frac{\phi (x)}{B^{\ast
}(x,n)-c_{n}A^{\ast }(x,n)}D\left( p_{n}^{N}(a;x)\right) \right) .%
\end{array}%
\end{equation*}%
Then a straightforward calculation yields

\setcounter{thm}{10}

\begin{thm}
The MOPS $\{p_{n}^{N}(a;x)\}_{n\geq 0}$ satisfies the second order linear
differential equation 
\begin{equation}
\mathcal{A}(x;n)(p_{n}^{N}(a;x))^{\prime \prime }+\mathcal{B}%
(x;n)(p_{n}^{N}(a;x))^{\prime }+\mathcal{C}(x;n)p_{n}^{N}(a;x)=0,
\label{Holon-Equ-Q-1}
\end{equation}%
where%
\begin{equation*}
\begin{array}{lll}
\mathcal{A}(x;n) & = & \displaystyle\frac{c_{n}\left[ \phi (x)\right] ^{2}}{%
B^{\ast }(x,n)-c_{n}A^{\ast }(x,n)},\vspace{0.3cm} \\ 
\mathcal{B}(x;n) & = & \displaystyle\frac{\phi (x)\left[ B(x,n)-B^{\ast
}(x,n)+c_{n}(\phi ^{\prime }(x)-A(x,n))\right] }{B^{\ast }(x,n)-c_{n}A^{\ast
}(x,n)}\vspace{0.3cm} \\ 
&  & \displaystyle\hspace{2.95cm}-\frac{c_{n}\phi (x)^{2}\left( B^{\ast
}(x,n)-c_{n}A^{\ast }(x,n)\right) ^{\prime }}{\left( B^{\ast
}(x,n)-c_{n}A^{\ast }(x,n)\right) ^{2}}\vspace{0.3cm} \\ 
\mathcal{C}(x;n) & = & \displaystyle\frac{A(x,n)B^{\ast }(x,n)-B(x,n)A^{\ast
}(x,n)}{B^{\ast }(x,n)-c_{n}A^{\ast }(x,n)}-\phi (x)D\left( \frac{B^{\ast
}(x,n)}{B^{\ast }(x,n)-c_{n}A^{\ast }(x,n)}\right) .%
\end{array}%
\end{equation*}
\end{thm}

A different approach to this differential equation appears in \cite%
{MarcMar92} using the fact that the Uvarov transform of a semiclassical
linear functional is again a semiclassical linear functional.

It is important to notice that for the electrostatic interpretation of the
zeros is enough to consider the polynomial coefficients of $%
(p_{n}^{N}(a;x))^{\prime\prime }$ and $(p_{n}^{N}(a;x))^{\prime }$. In fact,
it will come from the ratio 
\begin{equation*}
\frac{(p_{n}^{N}(a;x))^{\prime \prime }}{(p_{n}^{N}(a;x))^{\prime }}=-\frac{%
\mathcal{B}(x;n)}{\mathcal{A}(x;n)}
\end{equation*}
evaluated at the zeros of $p_{n}^{N}(a;x)$.

Let $(x_{n,k}^{N})$ be the zeros of $p_{n}^{N}(a;x)$. If we evaluate the
second-order linear differential equation (\ref{Holon-Equ-Q-1}) at $%
x_{n,k}^{N}$ then we obtain 
\begin{equation*}
\mathcal{A}(x_{n,k}^{N};n)(p_{n}^{N}(a;x_{n,k}^{N}))^{\prime \prime }+%
\mathcal{B}(x_{n,k}^{N};n)(p_{n}^{N}(a;x_{n,k}^{N}))^{\prime }=0.
\end{equation*}%
when $x_{n,k}^{N}$ is a zero of the polynomial $p_{n}^{N}(a;x)$.

Hence,%
\begin{equation}
\frac{(p_{n}^{N}(a;x_{n,k}^{N}))^{\prime \prime }}{%
(p_{n}^{N}(a;x_{n,k}^{N}))^{\prime }}=-\frac{\mathcal{B}(x_{n,k}^{N};n)}{%
\mathcal{A}(x_{n,k}^{N};n)}.  \label{electr_04}
\end{equation}

Substituting $\mathcal{A}(x_{n,k}^{N};n)$ and $\mathcal{B}(x_{n,k}^{N};n)$
in the right hand side of (\ref{electr_04}), we get%
\begin{eqnarray*}
&&\frac{(p_{n}^{N}(a;x_{n,k}^{N}))^{\prime \prime }}{%
(p_{n}^{N}(a;x_{n,k}^{N}))^{\prime }} \\
&=&\frac{(B^{\ast }(x_{n,k}^{N},n)-c_{n}A^{\ast }(x_{n,k}^{N},n))^{\prime }}{%
B^{\ast }(x_{n,k}^{N},n)-c_{n}A^{\ast }(x_{n,k}^{N},n)}+\frac{B^{\ast
}(x_{n,k}^{N},n)-B(x_{n,k}^{N},n)+c_{n}A(x_{n,k}^{N},n)-c_{n}\phi ^{\prime
}(x_{n,k}^{N})}{c_{n}\phi (x_{n,k}^{N})}.
\end{eqnarray*}

If we denote $Q(x):=B^{\ast }(x,n)-c_{n}A^{\ast }(x,n)$ and using (\ref%
{lema22}), (\ref{Coeff_4}), and (\ref{Coeff_5}), then we have%
\begin{equation*}
\begin{array}{lll}
Q(x) & = & \displaystyle B^{\ast }(x,n)-c_{n}A^{\ast }(x,n)\vspace{0.3cm} \\ 
& = & \displaystyle B(x,n)+c_{n}\left[ A(x,n-1)+\frac{x-\beta _{n-1}^{\ast }%
}{\gamma _{n-1}^{\ast }}B(x,n-1)-A(x,n)+\frac{c_{n}}{\gamma _{n-1}^{\ast }}%
B(x,n-1)\right] \vspace{0.3cm} \\ 
& = & \displaystyle B(x,n)+c_{n}\left[ -2A(x,n)+\left( A(x,n)+A(x,n-1)+\frac{%
(x-\beta _{n-1}^{\ast })}{\gamma _{n-1}^{\ast }}B(x,n-1)\right) \right. 
\vspace{0.3cm} \\ 
&  & \hspace{6.7cm}\left. +\displaystyle\frac{c_{n}}{\gamma _{n-1}^{\ast }}%
B(x,n-1)\right] \vspace{0.3cm} \\ 
& = & \displaystyle B(x,n)+c_{n}\left[ -2A(x,n)+\phi ^{\prime }(x)-\psi (x)+%
\frac{c_{n}}{\gamma _{n-1}^{\ast }}B(x,n-1)\right] ,%
\end{array}%
\end{equation*}%
i.e., 
\begin{equation}
Q(x)=\displaystyle B(x,n)+c_{n}\left[ -2A(x,n)+\phi ^{\prime }(x)-\psi (x)+%
\frac{c_{n}}{\gamma _{n-1}^{\ast }}B(x,n-1)\right] .  \label{Qnx}
\end{equation}

On the other hand, from (\ref{Qnx}) and (\ref{Coeff_4}), we obtain 
\begin{equation*}
\begin{array}{l}
B^{\ast }(x,n)-B(x,n)+c_{n}A(x,n)-c_{n}\phi ^{\prime }(x)\vspace{0.3cm} \\ 
=Q(x)+c_{n}A^{\ast }(x,n)-B(x,n)+c_{n}A(x,n)-c_{n}\phi ^{\prime }(x)\vspace{%
0.3cm} \\ 
=\displaystyle B(x,n)+c_{n}\left[ -2A(x,n)+\phi ^{\prime }(x)-\psi (x)+\frac{%
c_{n}}{\gamma _{n-1}^{\ast }}B(x,n-1)\right] \vspace{0.3cm} \\ 
+c_{n}\left( A(x,n)-\dfrac{c_{n}}{\gamma _{n-1}^{\ast }}B(x,n-1)\right)
-B(x,n)+c_{n}A(x,n)-c_{n}\phi ^{\prime }(x)\vspace{0.3cm} \\ 
=-c_{n}\psi (x).%
\end{array}%
\end{equation*}

Thus 
\begin{equation}
\displaystyle\frac{(p_{n}^{N}(a;x_{n,k}^{N}))^{\prime \prime }}{%
(p_{n}^{N}(a;x_{n,k}^{N}))^{\prime }}=D\left[ \ln Q(x)\right]
|_{x=x_{n,k}^{N}}-\frac{\psi (x_{n,k}^{N})}{\phi (x_{n,k}^{N})}.  \label{BA}
\end{equation}

We consider two external fields 
\begin{equation*}
-\int \frac{\psi (x)}{\phi (x)}dx\ \ \ \mbox{and}\ \ \ \ln Q(x),
\end{equation*}%
in such a way that the total external potential $V(x)$ is given by%
\begin{equation}
V(x)=-\int \frac{\psi (x)}{\phi (x)}dx+\ln Q(x).  \label{V}
\end{equation}

Let introduce a system of $n$ movable unit charges in $[a,\eta ]$ or $[\xi
,a] $, depending on the location of the point $a$ with respect to $%
C_{0}(\Sigma )=[\xi ,\eta ]$, in the presence of the external potential $%
V(x) $ of (\ref{V}). Let 
\begin{equation*}
x:=(x_{1},\ldots ,x_{n}),
\end{equation*}%
where $x_{1},\ldots ,x_{n}$ denote the positions of the particles. The total
energy of the system is%
\begin{equation}
E(x)=\sum_{k=1}^{n}V(x_{k})-2\sum_{1\leq j<k\leq n}\ln |x_{j}-x_{k}|.
\label{E}
\end{equation}

Let 
\begin{equation}
T(x):=\exp (-E(x))=\left[ \prod_{j=1}^{n}\frac{\exp (-\int \frac{\psi (x_{j})%
}{\phi (x_{j})}dx)}{Q(x_{j})}\right] \prod_{1\leq j<k\leq
n}(x_{j}-x_{k})^{2}.  \label{T}
\end{equation}

In order to find the critical points of $E(x)$ we will analyze the gradient
of $\ln T(x)$. Indeed, 
\begin{equation*}
\frac{\partial }{\partial x_{j}}\ln T(x)=0,\ \ j=1,\ldots ,n,
\end{equation*}%
i.e, 
\begin{equation}
-\frac{\partial }{\partial x_{j}}E(x)=0\Leftrightarrow \frac{\psi (x_{j})%
}{\phi (x_{j})}-\frac{Q^{\prime }(x_{j})}{Q(x_{j})}+2\sum_{1\leq k\leq
n,k\neq j}\frac{1}{x_{j}-x_{k}}=0,\ \ j=1,\ldots ,n.  \label{GT}
\end{equation}

Let 
\begin{equation*}
f(y):=(y-x_{1})\cdots (y-x_{n}).
\end{equation*}

Thus, 
\begin{equation*}
\frac{\psi (x_{j})}{\phi (x_{j})}-\frac{Q^{\prime }(x_{j})}{Q(x_{j})}+\frac{%
f^{\prime \prime }(x_{j})}{f^{\prime }(x_{j})}=0,\ \ j=1,\ldots ,n,
\end{equation*}%
or, equivalently, 
\begin{equation*}
f^{\prime \prime }(y)+\frac{\mathcal{B}(y;n)}{\mathcal{A}(y;n)}f^{\prime
}(y)=0,\ \ \ y=x_{1},\ldots ,x_{n}.
\end{equation*}

Therefore 
\begin{equation}
f^{\prime \prime }(y)+\frac{\mathcal{B}(y;n)}{\mathcal{A}(y;n)}f^{\prime
}(y)+\frac{\mathcal{C}(y;n)}{\mathcal{A}(y;n)}f(y)=0,\ \ \ y=x_{1},\ldots
,x_{n}.  \label{ed}
\end{equation}%
On the other hand, from (\ref{Holon-Equ-Q-1}) and (\ref{ed}) we get 
\begin{equation*}
f(y)= p_{n}^{N}(a;y),
\end{equation*}
and then the zeros of $p_{n}^{N}(a;y)$ satisfy (\ref{GT}).

\bigskip


\subsection{Electrostatic interpretation of the zeros of Laguerre type
orthogonal polynomials}

Firstly, we shall enumerate some useful basic properties of the Laguerre
classical monic polynomials $L_{n}^{\alpha}(x)$.

\begin{enumerate}
\item[i.] Let $u$ be the linear funcional 
\begin{equation*}
\langle u,p\rangle =\displaystyle\int_{0}^{+\infty }p(x)x^{\alpha
}e^{-x}dx,\quad \alpha >-1,\,\,p\in \mathbb{P}.
\end{equation*}%
So $u$ satisfies the Pearson differential equation 
\begin{equation*}
D(\sigma (x)u)=\tau (x)u,
\end{equation*}%
where 
\begin{equation}
\sigma (x)=x,\ \ \ \tau (x)=\alpha +1-x.  \label{PearsonLag}
\end{equation}

\item[ii.] For every $n\in \mathbb{N},$ 
\begin{equation}
\begin{array}{l}
L_{-1}^{\alpha }(x)=0,\ \ \ L_{0}^{\alpha }(x)=1,\vspace{0.3cm} \\ 
L_{n+1}^{\alpha }(x)=(x-\beta _{n})L_{n}^{\alpha }(x)-\gamma
_{n}L_{n-1}^{\alpha }(x),%
\end{array}
\label{RRLagu}
\end{equation}%
where 
\begin{equation*}
\beta _{n}=\beta _{n}^{\alpha }=2n+\alpha +1,\ \ \gamma _{n}=\gamma
_{n}^{\alpha }=n\left( n+\alpha \right) .
\end{equation*}


\item[iii.] For every $n\in \mathbb{N}$ 
\begin{equation}
L_{n}^{\alpha }(0)=(-1)^{n}\frac{\Gamma (n+\alpha +1)}{\Gamma (\alpha +1)}.
\label{L0}
\end{equation}

\item[iv.] For every $n\in \mathbb{N}$ 
\begin{equation}  \label{K00L}
K_{n}(0,0)=\displaystyle\frac{1}{n!}\frac{\Gamma (n+\alpha +2)}{\Gamma
(\alpha +2)\Gamma (\alpha+1)}.
\end{equation}

\item[v.] For every $n\in \mathbb{N}$ 
\begin{equation}  \label{StrctRelLag}
x\left[L_{n}^{\alpha}(x)\right]^{\prime
}=nL_{n}^{\alpha}(x)+n(n+\alpha)L_{n-1}^{\alpha}(x).
\end{equation}
\end{enumerate}

Now, we shall give an electrostatic interpretation for the zeros of Laguerre
type polynomials $L_{n}^{\alpha,N}(a;x)$ which are orthogonal with respect
to the measure $d\mu_N=x^{\alpha}e^{-x}dx+N\delta_{a}$, that is, they are
orthogonal with respect to the following inner product: 
\begin{equation*}
\langle p,q \rangle=\displaystyle\int_{0}^{+\infty}p(x)q(x)x^{%
\alpha}e^{-x}dx+Np(a)q(a),\ \ \ a\leq0.
\end{equation*}

We will analyze two cases:

\bigskip \textbf{1.} Firstly, we consider $a=0$. Thus, the polynomials $%
L_{n}^{\alpha,N}(0;x)$ are orthogonal with respect to 
\begin{equation*}
d\mu_N=x^{\alpha}e^{-x}dx+N\delta_{0}.
\end{equation*}

Now, observe that the polynomials $p_{n}^{\ast }(0;x)=L_{n}^{\alpha +1}(x)$
associated with the measure 
\begin{equation*}
d\mu ^{\ast }(x)=x^{\alpha +1}e^{-x}dx
\end{equation*}%
have Pearson's coefficients given by (see (\ref{Pea2}) and (\ref{PearsonLag}%
)) 
\begin{equation*}
\phi (x)=\sigma (x)=x,\ \ \ \psi (x)=\widetilde{\sigma }(x)+\tau (x)=\alpha
+2-x.
\end{equation*}%
%
%
%
%
%
%
%
On the other hand, from (\ref{StrctRelLag}), the structure relation (\ref%
{Struct-Rel}) reads 
\begin{equation*}
\phi (x)D\left( L_{n}^{\alpha +1}(x)\right) =A(x,n)L_{n}^{\alpha
+1}(x)+B(x,n)L_{n-1}^{\alpha +1}(x),
\end{equation*}%
where%
\begin{equation*}
\phi (x)=x,\ \ \ A(x,n)=n,\ \ \ B(x,n)=n+\alpha +1.
\end{equation*}%
In this case, the coefficients (\ref{cn}) and (\ref{TTRRK}) are given by 
\begin{eqnarray*}
\gamma _{n}^{\ast } &=&n(n+\alpha +1), \\
c_{n} &=&-\frac{1+NK_{n}(0,0)}{1+NK_{n-1}(0,0)}\frac{L_{n-1}^{\alpha }(0)}{%
L_{n}^{\alpha }(0)}n(n+\alpha ).
\end{eqnarray*}%
Using (\ref{L0}) and (\ref{K00L}), we obtain 
\begin{equation*}
\begin{array}{lll}
c_{n} & = & \dfrac{1+N\dfrac{\Gamma (n+\alpha +2)}{n!\Gamma (\alpha
+1)\Gamma (\alpha +2)}}{1+N\dfrac{\Gamma (n+\alpha +1)}{(n-1)!\Gamma (\alpha
+1)\Gamma (\alpha +2)}}\cdot n\vspace{0.3cm} \\ 
& = & \dfrac{n!\Gamma (\alpha +1)\Gamma (\alpha +2)+N\Gamma (n+\alpha +2)}{%
(n-1)!\Gamma (\alpha +1)\Gamma (\alpha +2)+N\Gamma (n+\alpha +1)}.%
\end{array}%
\end{equation*}%
As a conclusion, $Q(x)$ in (\ref{Qnx}) becomes%
\begin{equation*}
\begin{array}{lll}
Q\left( x\right) & = & \displaystyle B(x,n)+c_{n}\left[ -2A(x,n)+\phi
^{\prime }(x)-\psi (x)+\frac{c_{n}}{\gamma _{n-1}^{\ast }}B(x,n-1)\right] 
\vspace{0.3cm} \\ 
& = & n(n+\alpha +1)+c_{n}\left[ -2n+1-(\alpha +2-x)+c_{n}\right] \vspace{%
0.3cm} \\ 
& = & n(n+\alpha +1)-c_{n}\left( 2n+1+\alpha -c_{n}\right) +c_{n}x%
\end{array}%
\end{equation*}%
and its zero will be denoted by%
\begin{equation}
u_{n}=\left( 2n+1+\alpha -c_{n}\right) -\frac{n(n+\alpha +1)}{c_{n}}.
\label{zero-caso-1}
\end{equation}%
Now, it is easily to see that $0<c_{n}<n+\alpha +1$. Thus, $Q(0)<0$ and it
implies that $u_{n}>0$.

Taking into account%
\begin{equation*}
\Gamma (z)\sim \sqrt{\frac{2\pi }{z}}e^{-z}z^{z}=\sqrt{2\pi }e^{-z}z^{z-%
\frac{1}{2}},
\end{equation*}%
it is easy to see that 
\begin{eqnarray}
\frac{\Gamma (n+\alpha +1)}{\Gamma (n+1)} &\sim &\frac{\sqrt{2\pi }%
e^{-n-1}(n+\alpha +1)^{n+\frac{1}{2}}(n+\alpha +1)^{\alpha }e^{-\alpha }}{%
\sqrt{2\pi }e^{-n-1}(n+1)^{n+\frac{1}{2}}}  \notag \\
&=&(n+\alpha +1)^{\alpha }e^{-\alpha }\left( 1+\frac{\alpha }{n+1}\right)
^{n+\frac{1}{2}}  \notag \\
&\sim &(n+\alpha +1)^{\alpha }\sim n^{\alpha }.  \label{ratio-Gamma-1}
\end{eqnarray}%
Thus 
\begin{eqnarray}
1+N\frac{\Gamma (n+\alpha +1)}{(n-1)!\Gamma (\alpha +1)\Gamma (\alpha +2)}
&=&1+\frac{\Gamma (n+\alpha +1)}{\Gamma (n)}\frac{N}{\Gamma (\alpha
+1)\Gamma (\alpha +2)}  \notag \\
&\sim &1+\frac{Nn^{\alpha +1}}{\Gamma (\alpha +1)\Gamma (\alpha +2)}  \notag
\\
&\sim &\frac{Nn^{\alpha +1}}{\Gamma (\alpha +1)\Gamma (\alpha +2)}.
\label{ratio-Gamma-2}
\end{eqnarray}%
From (\ref{zero-caso-1}), we have 
\begin{eqnarray*}
u_{n} &=&\left( 2n+1+\alpha -c_{n}\right) -\frac{n(n+\alpha +1)}{c_{n}} \\
&=&\left( 2n+1+\alpha \right) -\dfrac{n+\alpha +1-\alpha -1+N\dfrac{%
(n+\alpha +1)\Gamma (n+\alpha +1)}{\left( n-1\right) !\Gamma (\alpha
+1)\Gamma (\alpha +2)}}{1+N\dfrac{\Gamma (n+\alpha +1)}{(n-1)!\Gamma (\alpha
+1)\Gamma (\alpha +2)}} \\
&&-(n+\alpha +1)\dfrac{1+N\dfrac{\Gamma (n+\alpha +1)}{(n-1)!\Gamma (\alpha
+1)\Gamma (\alpha +2)}}{1+N\dfrac{\Gamma (n+\alpha +2)}{n!\Gamma (\alpha
+1)\Gamma (\alpha +2)}}.
\end{eqnarray*}%
Then, after some computations, and using (\ref{ratio-Gamma-1}) and (\ref%
{ratio-Gamma-2}) we can estimate its behavior with respect to $N$ and $n$: 
\begin{eqnarray*}
u_{n} &=&\dfrac{\alpha +1}{\left( 1+N\dfrac{\Gamma (n+\alpha +1)}{%
(n-1)!\Gamma (\alpha +1)\Gamma (\alpha +2)}\right) }-\dfrac{\alpha +1}{%
\left( 1+N\dfrac{\Gamma (n+\alpha +2)}{n!\Gamma (\alpha +1)\Gamma (\alpha +2)%
}\right) } \\
&=&\dfrac{\left( \alpha +1\right) ^{2}}{\left( 1+N\frac{\Gamma (n+\alpha +1)%
}{\Gamma (n+1)}\frac{n}{\Gamma (\alpha +1)\Gamma (\alpha +2)}\right) \left(
1+N\frac{\Gamma (n+\alpha +1)}{\Gamma (n+1)}\frac{(n+\alpha +1)}{\Gamma
(\alpha +1)\Gamma (\alpha +2)}\right) } \\
&&\cdot \frac{N\Gamma (n+\alpha +1)}{\left( n!\Gamma (\alpha +1)\Gamma
(\alpha +2)\right) } \\
&\sim &\dfrac{\left( \alpha +1\right) ^{2}}{\left( N\frac{n^{\alpha +1}}{%
\Gamma (\alpha +1)\Gamma (\alpha +2)}\right) ^{2}}\cdot \dfrac{Nn^{\alpha }}{%
\Gamma (\alpha +1)\Gamma (\alpha +2)} \\
&=&\dfrac{\left( \alpha +1\right) \left[ \Gamma (\alpha +2)\right] ^{2}}{N}%
n^{-\alpha -2}
\end{eqnarray*}

The electrostatic interpretation of the distribution of zeros means that we
have an equilibrium position under the presence of an external
potential 
\begin{equation*}
\text{ln }Q(x)+\text{ln }x^{\alpha +2}e^{-x},
\end{equation*}%
where the first term represents a short range potential corresponding to a
unit charge located at $u_{n}$ and the second one is a long range potential
associated with the weight function (see also \cite{Ism00-B} and \cite{Ism05}%
). \bigskip

\textbf{2.} Now, we will consider $a<0$. In this case $d\mu^{%
\ast}(x)=(x-a)x^{\alpha }e^{-x}dx$. Thus,

\setcounter{prop}{4}

\begin{prop}
The structure relation \emph{(\ref{Struct-Rel})} for the measure 
\begin{equation*}
d\mu ^{\ast }(x)=(x-a)x^{\alpha }e^{-x}dx,\ \ a<0,
\end{equation*}%
is%
\begin{equation*}
\phi (x)D(p_{n}^{\ast }(a;x))=A(x,n)p_{n}^{\ast }(a;x)+B(x,n)p_{n-1}^{\ast
}(a;x),
\end{equation*}%
where 
\begin{equation*}
\begin{array}{lll}
\phi (x) & = & (x-a)x,\vspace{0.3cm} \\ 
A(x,n) & = & \displaystyle n\left[ x-\left( n+1+a_{n}\right) \left( 1+\frac{%
n+\alpha }{a_{n-1}}\right) \right] ,\vspace{0.3cm} \\ 
B(x,n) & = & \displaystyle\frac{n\left( n+\alpha \right) }{a_{n-1}}\left[
a_{n}x-\left( n+1+a_{n}\right) \left( n+1+a_{n}+\alpha \right) \right] .%
\end{array}%
\end{equation*}
\end{prop}

\begin{proof}
From (\ref{p1}) we get 
\begin{equation*}
\left( x-a\right) p_{n}^{\ast }(a;x)=L_{n+1}^{\alpha }(x)-\dfrac{%
L_{n+1}^{\alpha }(a)}{L_{n}^{\alpha }(a)}L_{n}^{\alpha }(x).
\end{equation*}%
Taking derivatives with respect to $x$ in both hand sides of the above
expression, and multiplying the resulting expression by $x$, we see that%
\begin{equation*}
xp_{n}^{\ast }(a;x)+x\left( x-a\right) D\left( p_{n}^{\ast }(a;x)\right)
=xD\left( L_{n+1}^{\alpha }(x)\right) -\frac{L_{n+1}^{\alpha }(a)}{%
L_{n}^{\alpha }(a)}xD\left( L_{n}^{\alpha }(x)\right) .
\end{equation*}%
Using the structure relation (\ref{StrctRelLag}) for Laguerre polynomials,
we obtain%
\begin{equation*}
\begin{array}{lll}
xp_{n}^{\ast }(a;x)+x\left( x-a\right) D\left( p_{n}^{\ast }(a;x)\right) & =
& \left( n+1\right) L_{n+1}^{\alpha }(x)+\left( n+1\right) \left( n+1+\alpha
\right) L_{n}^{\alpha }(x)\vspace{0.3cm} \\ 
&  & -\displaystyle\frac{L_{n+1}^{\alpha }(a)}{L_{n}^{\alpha }(a)}\left[
nL_{n}^{\alpha }(x)+n\left( n+\alpha \right) L_{n-1}^{\alpha }(x)\right] .%
\end{array}%
\end{equation*}%
Now, from (\ref{RRLagu}), we get 
\begin{equation*}
\begin{array}{l}
xp_{n}^{\ast }(a;x)+x\left( x-a\right) Dp_{n}^{\ast }(a;x)\vspace{0.3cm} \\ 
\displaystyle=\left( n+1\right) xL_{n}^{\alpha }(x)-\left( n+1\right) \left(
2n+1+\alpha \right) L_{n}^{\alpha }(x)-\left( n+1\right) n\left( n+\alpha
\right) L_{n-1}^{\alpha }(x)\vspace{0.3cm} \\ 
\displaystyle\hspace{0.5cm}+\left[ \left( n+1\right) \left( n+1+\alpha
\right) -\frac{nL_{n+1}^{\alpha }(a)}{L_{n}^{\alpha }(a)}\right]
L_{n}^{\alpha }(x)-\frac{L_{n+1}^{\alpha }(a)}{L_{n}^{\alpha }(a)}n\left(
n+\alpha \right) L_{n-1}^{\alpha }(x)\vspace{0.3cm} \\ 
\displaystyle=\left[ \left( n+1\right) x-\left( n+1\right) \left(
2n+1+\alpha \right) +\left( n+1\right) \left( n+1+\alpha \right) -\frac{%
nL_{n+1}^{\alpha }(a)}{L_{n}^{\alpha }(a)}\right] L_{n}^{\alpha }(x)\vspace{%
0.3cm} \\ 
\displaystyle\hspace{0.5cm}-\left( n\left( n+\alpha \right) \left(
n+1\right) +n\left( n+\alpha \right) \frac{L_{n+1}^{\alpha }(a)}{%
L_{n}^{\alpha }(a)}\right) L_{n-1}^{\alpha }(x)\vspace{0.3cm} \\ 
\displaystyle=\left[ \left( n+1\right) \left( x-n\right) -\frac{%
nL_{n+1}^{\alpha }(a)}{L_{n}^{\alpha }(a)}\right] L_{n}^{\alpha }(x)\vspace{%
0.3cm} \\ 
\displaystyle\hspace{0.5cm}-\left( n\left( n+1\right) \left( n+\alpha
\right) +n\left( n+\alpha \right) \frac{L_{n+1}^{\alpha }(a)}{L_{n}^{\alpha
}(a)}\right) L_{n-1}^{\alpha }(x).%
\end{array}%
\end{equation*}%
Now, using the notation%
\begin{equation}
a_{n}=\dfrac{L_{n+1}^{\alpha }\left( a\right) }{L_{n}^{\alpha }\left(
a\right) },  \label{ratios-1}
\end{equation}%
and, from (\ref{pneqKnKn-1}), 
\begin{equation*}
L_{n}^{\alpha }(x)=p_{n}^{\ast }(a;x)-\frac{\gamma _{n}}{a_{n-1}}%
p_{n-1}^{\ast }(a;x),
\end{equation*}%
we have 
\begin{equation*}
\begin{array}{l}
xp_{n}^{\ast }(a;x)+x\left( x-a\right) D\left( p_{n}^{\ast }(a;x)\right) 
\vspace{0.3cm} \\ 
\displaystyle=\left( \left( x-n\right) \left( n+1\right) -na_{n}\right)
L_{n}^{\alpha }(x)-n\left( n+\alpha \right) \left( n+1+a_{n}\right)
L_{n-1}^{\alpha }(x)\vspace{0.3cm} \\ 
\displaystyle=\left( \left( x-n\right) \left( n+1\right) -na_{n}\right) %
\left[ p_{n}^{\ast }(a;x)-\gamma _{n}\frac{1}{a_{n-1}}p_{n-1}^{\ast }(a;x)%
\right] \vspace{0.3cm} \\ 
\displaystyle\hspace{0.5cm}-n\left( n+\alpha \right) \left( n+1+a_{n}\right) %
\left[ p_{n-1}^{\ast }(a;x)-\gamma _{n-1}\frac{1}{a_{n-2}}p_{n-2}^{\ast
}(a;x)\right] \vspace{0.3cm} \\ 
\displaystyle=\left( \left( x-n\right) \left( n+1\right) -na_{n}\right)
p_{n}^{\ast }(a;x)\vspace{0.3cm} \\ 
\displaystyle\hspace{0.5cm}-\left[ \left( \left( n+1\right) \left(
x-n\right) -na_{n}\right) \frac{n\left( n+\alpha \right) }{a_{n-1}}+n\left(
n+\alpha \right) \left( n+1+a_{n}\right) \right] p_{n-1}^{\ast }(a;x)\vspace{%
0.3cm} \\ 
\displaystyle\hspace{0.5cm}+n\left( n-1\right) \left( n+\alpha \right)
\left( n-1+\alpha \right) \left( n+1+a_{n}\right) \frac{p_{n-2}^{\ast }(a;x)%
}{a_{n-2}}.%
\end{array}%
\end{equation*}%
Now, using the TTRR (\ref{TTRRK}) for monic kernels, and according to (\ref%
{coef-RRTTK}) 
\begin{equation*}
\frac{\gamma _{n-1}^{\ast }}{\gamma _{n-1}}=\dfrac{L_{n}^{\alpha }(a)}{%
L_{n-1}^{\alpha }(a)}\dfrac{L_{n-2}^{\alpha }(a)}{L_{n-1}^{\alpha }(a)}
\end{equation*}%
we obtain 
\begin{eqnarray*}
\frac{a_{n-1}}{a_{n-2}} &=&\frac{\gamma _{n-1}^{\ast }}{\left( n-1\right)
\left( n-1+\alpha \right) } \\
\frac{\left( n-1\right) \left( n-1+\alpha \right) }{a_{n-2}} &=&\frac{\gamma
_{n-1}^{\ast }}{a_{n-1}}
\end{eqnarray*}%
and then 
\begin{equation*}
\begin{array}{l}
xp_{n}^{\ast }(a;x)+x\left( x-a\right) D\left( p_{n}^{\ast }(a;x)\right) 
\vspace{0.3cm} \\ 
\displaystyle=\left( \left( x-n\right) \left( n+1\right) -na_{n}\right)
p_{n}^{\ast }(a;x)\vspace{0.3cm} \\ 
\displaystyle\hspace{0.5cm}-\left[ \left( \left( n+1\right) \left(
x-n\right) -na_{n}\right) \frac{n\left( n+\alpha \right) }{a_{n-1}}+n\left(
n+\alpha \right) \left( n+1+a_{n}\right) \right] p_{n-1}^{\ast }(a;x)\vspace{%
0.3cm} \\ 
\displaystyle\hspace{0.5cm}+n\left( n+\alpha \right) \left( n+1+a_{n}\right) 
\frac{\gamma _{n-1}^{\ast }p_{n-2}^{\ast }(a;x)}{a_{n-1}}\vspace{0.3cm} \\ 
\displaystyle=\left[ \left( \left( x-n\right) \left( n+1\right)
-na_{n}\right) -n\left( n+\alpha \right) \left( n+1+a_{n}\right) \frac{1}{%
a_{n-1}}\right] p_{n}^{\ast }(a;x)\vspace{0.3cm} \\ 
\displaystyle\hspace{0.5cm}-n\left( n+\alpha \right) \left[ \left( \left(
n+1\right) \left( x-n\right) -na_{n}\right) \frac{1}{a_{n-1}}\right. \vspace{%
0.3cm} \\ 
\displaystyle\hspace{0.5cm}+\left. \left( n+1+a_{n}\right) -\left(
n+1+a_{n}\right) \left( x-\beta _{n-1}^{\ast }\right) \frac{1}{a_{n-1}}%
\right] p_{n-1}^{\ast }(a;x)\vspace{0.3cm} \\ 
\end{array}%
\end{equation*}

\begin{equation*}
\begin{array}{l}
\displaystyle=\left[ \left( \left( n+1\right) \left( x-n\right)
-na_{n}\right) -\frac{n\left( n+\alpha \right) \left( n+1+a_{n}\right) }{%
a_{n-1}}\right] p_{n}^{\ast }(a;x)\vspace{0.3cm} \\ 
\displaystyle\hspace{0.5cm}+n\left( n+\alpha \right) \left[ \frac{1}{a_{n-1}}%
\left( n+1+a_{n}\right) \left( x-\beta _{n-1}^{\ast }\right) \right. \vspace{%
0.3cm} \\ 
\displaystyle\hspace{0.5cm}-\left. \frac{1}{a_{n-1}}\left( \left( n+1\right)
\left( x-n\right) -na_{n}\right) -\left( n+1+a_{n}\right) \right]
p_{n-1}^{\ast }(a;x).%
\end{array}%
\end{equation*}%
Therefore 
\begin{equation*}
\phi (x)D(p_{n}^{\ast }(a;x))=A(x,n)p_{n}^{\ast }(a;x)+B(x,n)p_{n-1}^{\ast
}(a;x),
\end{equation*}%
where 
\begin{eqnarray*}
\phi (x) &=&x\left( x-a\right) \\
A(x,n) &=&\left( \left( n+1\right) \left( x-n\right) -na_{n}\right) -\frac{%
n\left( n+\alpha \right) \left( n+1+a_{n}\right) }{a_{n-1}}-x \\
B(x,n) &=&n\left( n+\alpha \right) \left[ \frac{1}{a_{n-1}}\left(
n+1+a_{n}\right) \left( x-\beta _{n-1}^{\ast }\right) \right. \\
&&-\left. \frac{1}{a_{n-1}}\left( \left( n+1\right) \left( x-n\right)
-na_{n}\right) -\left( n+1+a_{n}\right) \right] .
\end{eqnarray*}%
Simplifying these expressions we have 
\begin{equation*}
\begin{array}{lll}
\displaystyle A(x,n) & = & \displaystyle nx-n\left( n+1\right) -na_{n}-\frac{%
n\left( n+\alpha \right) \left( n+1+a_{n}\right) }{a_{n-1}}\vspace{0.3cm} \\ 
& = & \displaystyle n\left[ x-\left( n+1+a_{n}\right) \left( 1+\frac{%
n+\alpha }{a_{n-1}}\right) \right]%
\end{array}%
\end{equation*}%
and 
\begin{equation*}
\begin{array}{lll}
B(x.n) & = & \displaystyle n\left( n+\alpha \right) \left[ \frac{%
xa_{n}+n+n^{2}}{a_{n-1}}+\frac{na_{n}}{a_{n-1}}-\frac{\left(
n+1+a_{n}\right) }{a_{n-1}}\beta _{n-1}^{\ast }-\left( n+1+a_{n}\right) %
\right] \vspace{0.3cm} \\ 
& = & \displaystyle n\left( n+\alpha \right) \left[ \frac{a_{n}}{a_{n-1}}x+%
\frac{n+1+a_{n}}{a_{n-1}}\left( n-\beta _{n-1}^{\ast }\right) -\left(
n+1+a_{n}\right) \right] .%
\end{array}%
\end{equation*}%
In the above expression, using again (\ref{coef-RRTTK}) 
\begin{equation*}
\beta _{n}^{\ast }=\beta _{n+1}+a_{n+1}-a_{n}=2n+\alpha +3+a_{n+1}-a_{n}
\end{equation*}%
we obtain 
\begin{equation*}
\displaystyle B(x,n)=\displaystyle\frac{n\left( n+\alpha \right) }{a_{n-1}}%
\left[ a_{n}x-\left( n+1+a_{n}\right) \left( n+1+a_{n}+\alpha \right) \right]
.
\end{equation*}
\end{proof}

This is an alternative approach to the method described in \cite{MarcRon89}.
\bigskip

Notice that the Pearson equation for the linear functional associated with
the measure becomes%
\begin{equation*}
D\left( \phi u^{\ast }\right) =\psi u^{\ast }
\end{equation*}%
and (see (\ref{Pea1}) and (\ref{PearsonLag})) 
\begin{eqnarray*}
\phi (x) &=&(x-a)\sigma (x)=(x-a)x, \\
\psi (x) &=&2\sigma (x)+(x-a)\tau (x)=2x+\left( x-a\right) \left( \alpha
+1-x\right) .
\end{eqnarray*}

According to (\ref{coef-RRTTK}) and (\ref{coef-RRTTK-2}),%
\begin{eqnarray*}
\beta _{n-1}^{\ast } &=&\beta _{n}+a_{n}-a_{n-1}, \\
\gamma _{n-1}^{\ast } &=&\dfrac{a_{n-1}}{a_{n-2}}\gamma _{n-1}.
\end{eqnarray*}%
This means that $Q(x)$ in (\ref{Qnx}) is the following quadratic polynomial 
\begin{equation*}
\begin{array}{lll}
Q\left( x\right) & = & \displaystyle B(x,n)+c_{n}\left[ -2A(x,n)+\phi
^{\prime }(x)-\psi (x)+\frac{c_{n}}{\gamma _{n-1}^{\ast }}B(x,n-1)\right] 
\vspace{0.3cm} \\ 
& = & \displaystyle\gamma _{n}\frac{a_{n}}{a_{n-1}}x+\gamma _{n}\frac{%
n+1+a_{n}}{a_{n-1}}\left( n-\left( \beta _{n}+a_{n}-a_{n-1}\right) \right) 
\vspace{0.3cm} \\ 
&  & -\displaystyle\gamma _{n}\left( n+1+a_{n}\right) -2nc_{n}x+2c_{n}\left(
n+1+a_{n}\right) \left( n+\frac{\gamma _{n}}{a_{n-1}}\right) \vspace{0.3cm}
\\ 
&  & \displaystyle+c_{n}x^{2}-c_{n}\left( a+\alpha +1\right) x+c_{n}\left(
a\left( \alpha +1\right) -a\right) +c_{n}^{2}x\vspace{0.3cm} \\ 
&  & \displaystyle+c_{n}^{2}\frac{n+a_{n-1}}{a_{n-1}}\left( n-1-\beta
_{n-1}-a_{n-1}+a_{n-2}\right) -c_{n}^{2}\frac{a_{n-2}}{a_{n-1}}\left(
n+a_{n-1}\right) .%
\end{array}%
\end{equation*}%
After some tedious computation the above expression becomes%
\begin{eqnarray*}
Q\left( x\right) &=&c_{n}x^{2}+\left( \gamma _{n}\frac{a_{n}}{a_{n-1}}%
-2nc_{n}+c_{n}^{2}-c_{n}\left( a+\alpha +1\right) \right) x \\
&&+\left( n+1+a_{n}\right) \left[ \left( 2c_{n}n-\gamma _{n}\right) +\frac{%
\gamma _{n}}{a_{n-1}}\left( n-\beta _{n}-a_{n}+a_{n-1}+2c_{n}\right) \right]
\\
&&+c_{n}^{2}\frac{\left( n+a_{n-1}\right) }{a_{n-1}}\left( n-1-2n+1-\alpha
-a_{n-1}\right) +ac_{n}\alpha ,
\end{eqnarray*}%
i.e., 
\begin{equation*}
Q\left( x\right) =c_{n}x^{2}+r_{n}x+s_{n}
\end{equation*}%
with%
\begin{eqnarray*}
r_{n} &=&n\left( n+\alpha \right) \frac{a_{n}}{a_{n-1}}+c_{n}^{2}-c_{n}%
\left( a+\alpha +1+2n\right) \\
&=&\left( c_{n}+a_{n}\right) \left( c_{n}-a_{n}\right) -\left(
c_{n}-a_{n}\right) a-\left( c_{n}+a_{n}\right) \left( 2n+\alpha +1\right)
\end{eqnarray*}%
and 
\begin{eqnarray*}
s_{n} &=&\left( n+1+a_{n}\right) \left[ \left( n+1+a_{n}+\alpha \right)
\left( 2n+1+a_{n}+\alpha -a-2c_{n}\right) +2ac_{n}\right] \\
&&+a\alpha c_{n}+c_{n}^{2}\left( a_{n}-a_{n-1}+1-a\right) .
\end{eqnarray*}

The zeros of this polynomial are%
\begin{eqnarray*}
z_{1,n} &=&-\frac{1}{2c_{n}}\left( r_{n}+\sqrt{r_{n}^{2}-4s_{n}c_{n}}\right)
, \\
z_{2,n} &=&-\frac{1}{2c_{n}}\left( r_{n}-\sqrt{r_{n}^{2}-4s_{n}c_{n}}\right)
.
\end{eqnarray*}

Taking into account%
\begin{eqnarray*}
\frac{\psi }{\phi } &=&\frac{2\sigma (x)+(x-a)\tau (x)}{(x-a)\sigma (x)} \\
&=&\frac{2}{x-a}+\frac{\tau (x)}{\sigma (x)} \\
&=&\frac{2}{x-a}+\frac{\alpha +1-x}{x} \\
&=&\frac{2}{x-a}+\frac{\alpha +1}{x}-1,
\end{eqnarray*}%
the electrostatic interpretation means that the equilibrium position for the
zeros under the presence of an external potential%
\begin{equation*}
\text{ln }Q(x)+\text{ln }\left( x-a\right) ^{2}x^{\alpha +1}e^{-x},
\end{equation*}%
where the first one is a short range potential corresponding to two unit
charges located at $z_{1,n}$ and $z_{2,n}$ and the second one is a long
range potential associated with a polynomial perturbation of the weight
function. \bigskip 


\subsection{Electrostatic interpretation for the zeros of Jacobi type
orthogonal polynomials}

We shall use some basic properties of the Jacobi classical monic polynomials 
$P_{n}^{\alpha,\beta}(x)$.

\begin{enumerate}
\item[i.] Let $u$ be the linear funcional 
\begin{equation*}
\langle u,p\rangle =\displaystyle\int_{-1}^{1}p(x)(1-x)^{\alpha
}(1+x)^{\beta }dx,\quad \alpha ,\beta >-1,\,\,p\in \mathbb{P}.
\end{equation*}%
So $u$ satisfies the Pearson differential equation 
\begin{equation*}
D(\sigma (x)u)=\tau (x)u,
\end{equation*}%
where 
\begin{equation}
\sigma (x)=1-x^{2},\ \ \ \tau (x)=(\beta -\alpha )-(\alpha +\beta +2)x.
\label{PearsonJac}
\end{equation}

\item[ii.] For every $n\in \mathbb{N},$ 
\begin{equation}
\begin{array}{l}
P_{-1}^{\alpha ,\beta }(x)=0,\ \ \ P_{0}^{\alpha ,\beta }(x)=1,\vspace{0.3cm}
\\ 
P_{n+1}^{\alpha ,\beta }(x)=(x-\beta _{n})P_{n}^{\alpha ,\beta }(x)-\gamma
_{n}P_{n-1}^{\alpha ,\beta }(x),%
\end{array}
\label{RRJac}
\end{equation}%
where 
\begin{equation*}
\begin{array}{l}
\beta _{n}=\beta _{n}^{\alpha ,\beta }=\displaystyle\frac{\beta ^{2}-\alpha
^{2}}{(2n+\alpha +\beta )(2n+\alpha +\beta +2)}%
\end{array}%
\end{equation*}%
and 
\begin{equation*}
\begin{array}{l}
\gamma _{n}=\gamma _{n}^{\alpha ,\beta }=\displaystyle\frac{4n(n+\alpha
)(n+\beta )(n+\alpha +\beta )}{(2n+\alpha +\beta -1)(2n+\alpha +\beta
)^{2}(2n+\alpha +\beta +1)}.%
\end{array}%
\end{equation*}


\item[iii.] For every $n\in \mathbb{N}$ 
\begin{equation}
P_{n}^{\alpha ,\beta }(-1)=\displaystyle\frac{(-1)^{n}2^{n}\Gamma (n+\beta
+1)\Gamma (n+\alpha +\beta +1)}{\Gamma (\beta +1)\Gamma (2n+\alpha +\beta +1)%
}.  \label{P1}
\end{equation}

\item[iv.] For every $n\in \mathbb{N}$ 
\begin{equation}
K_{n-1}(-1,-1)=\frac{1}{2^{\alpha +\beta +1}}\frac{\Gamma (n+\beta +1)\Gamma
(n+\alpha +\beta +1)}{\Gamma (n)\Gamma (\beta +1)\Gamma (\beta +2)\Gamma
(n+\alpha )}.  \label{K00J}
\end{equation}

\item[v.] For every $n\in \mathbb{N}$ 
\begin{equation}  \label{StrctRelJac}
\begin{array}{lll}
(1-x^2)D\left( P_{n}^{\alpha ,\beta}(x)\right) & = & \displaystyle\frac{%
-n[\beta -\alpha+(2n+\alpha +\beta)x]}{2n+\alpha +\beta}P_{n}^{\alpha,%
\beta}(x)\vspace{0.3cm} \\ 
&  & +\displaystyle\frac{4n(n+\alpha )(n+\beta)(n+\alpha +\beta)}{(2n+\alpha
+\beta)^{2}(2n+\alpha +\beta -1)}P_{n-1}^{\alpha ,\beta}(x).%
\end{array}%
\end{equation}
\end{enumerate}

Now, we shall give an electrostatic interpretation for the zeros of Jacobi
type polynomials $P_{n}^{\alpha ,\beta ,N}(a;x)$ which are orthogonal with
respect to the measure $d\mu _{N}=(1-x)^{\alpha }(1+x)^{\beta }dx+N\delta
_{a}$, that is, they are orthogonal with respect the following inner
product: 
\begin{equation*}
\langle p,q\rangle =\displaystyle\int_{-1}^{1}p(x)q(x)(1-x)^{\alpha
}(1+x)^{\beta }dx+Np(a)q(a),\ \ \ a\not\in (-1,1),\ N\geq 0.
\end{equation*}

We will analyze two cases: \bigskip

\textbf{1.} Firstly, we consider $a=-1$. Thus, the polynomials $%
P_{n}^{\alpha ,\beta ,N}(-1;x)$ are orthogonal with respect to 
\begin{equation*}
d\mu _{N}=(1-x)^{\alpha }(1+x)^{\beta }dx+N\delta _{-1}.
\end{equation*}%
Therefore, the polynomials $p_{n}^{\ast }(-1;x)=P_{n}^{\alpha ,\beta +1}(x)$
associated with the measure 
\begin{equation*}
d\mu ^{\ast }(x)=(x-(-1))d\mu (x)=(1-x)^{\alpha }(1+x)^{\beta +1}dx
\end{equation*}%
have Pearson's coefficients given by (see (\ref{Pea2}) and (\ref{PearsonJac}%
)) 
\begin{equation*}
\phi (x)=\sigma (x)=1-x^{2},\ \ \ \psi (x)=\widetilde{\sigma }(x)+\tau
(x)=\left( \beta -\alpha +1\right) -\left( \alpha +\beta +3\right) x.
\end{equation*}%
%
%
%
%
%
%
%
%
%
%
%
%
%
%
%
%
%
%
%
%
%
%
%
On the other hand, from (\ref{StrctRelJac}), the structure relation (\ref%
{Struct-Rel}) reads%
\begin{equation*}
\phi (x)D\left( P_{n}^{\alpha ,\beta +1}(x)\right) =A(x,n)P_{n}^{\alpha
,\beta +1}(x)+B(x,n)P_{n-1}^{\alpha ,\beta +1}(x),
\end{equation*}%
where

\begin{equation*}
\begin{array}{l}
A(x,n)=\displaystyle\frac{-n[\beta -\alpha +1+(2n+\alpha +\beta +1)x]}{%
2n+\alpha +\beta +1},\vspace{0.3cm} \\ 
B(x,n)=\displaystyle\frac{4n(n+\alpha )(n+\beta +1)(n+\alpha +\beta +1)}{%
(2n+\alpha +\beta +1)^{2}(2n+\alpha +\beta )}.%
\end{array}%
\end{equation*}%
The coefficient $\gamma _{n}^{\ast }$ in (\ref{TTRRK}) when $p_{n}^{\ast
}(-1;x)=P_{n}^{\alpha ,\beta +1}(x)$ is 
\begin{equation*}
\begin{array}{l}
\gamma _{n}^{\ast }=\gamma _{n}^{\alpha ,\beta +1}=\displaystyle\frac{%
4n(n+\alpha )(n+\beta +1)(n+\alpha +\beta +1)}{(2n+\alpha +\beta )(2n+\alpha
+\beta +1)^{2}(2n+\alpha +\beta +2)}.%
\end{array}%
\end{equation*}

Now, we will find the coefficient $c_{n}$ in (\ref{02nd-Rel}). Using (\ref%
{P1}) and (\ref{K00J}) it follows that%
\begin{eqnarray*}
c_{n} &=&-\frac{1+NK_{n}(-1,-1)}{1+NK_{n-1}(-1,-1)}\frac{P_{n-1}^{\alpha
,\beta }(-1)}{P_{n}^{\alpha ,\beta }(-1)}\cdot \gamma _{n}^{\alpha ,\beta }
\\
&=&\frac{1+NK_{n}(-1,-1)}{1+NK_{n-1}(-1,-1)}\frac{2n(n+\alpha )}{(2n+\alpha
+\beta )(2n+\alpha +\beta +1)}>0,
\end{eqnarray*}%
and, finally, we get%
\begin{eqnarray*}
Q(x) &=&B(x,n)+c_{n}\left[ (2n+\alpha +\beta )c_{n}-\frac{(\alpha +\beta
+1)(\beta -\alpha +1)}{2n+\alpha +\beta +1}\right] \\
&&+(2n+\alpha +\beta +1)c_{n}x.
\end{eqnarray*}

Now, we will show that the zero of $Q(x)$ belongs to $(-1,1)$. In fact,
observe that 
\begin{equation*}
\begin{array}{lll}
Q(1) & = & \displaystyle B(x,n)+c_{n}\left[ (2n+\alpha +\beta )c_{n}-\frac{%
(\alpha +\beta +1)(\beta -\alpha +1)}{2n+\alpha +\beta +1}\right] \vspace{%
0.3cm} \\ 
&  & +(2n+\alpha +\beta +1)c_{n}\vspace{0.3cm} \\ 
& = & \displaystyle B(x,n)+c_{n}\left[ (2n+\alpha +\beta )c_{n}+\frac{%
2(2n(n+\alpha +\beta +1)+\alpha (\alpha +\beta +1))}{2n+\alpha +\beta +1}%
\right] \vspace{0.3cm} \\ 
& > & 0,%
\end{array}%
\end{equation*}%
and, after some tedious calculations,%
\begin{equation*}
\begin{array}{lll}
Q(-1) & = & \displaystyle\frac{-2^{\alpha +\beta +3}(\beta +1)\Gamma
(n)\Gamma (n+\alpha )\Gamma (\beta +2)^{2}\Gamma (n+\beta +1)\Gamma
(n+\alpha +\beta +1)N}{2n+\alpha +\beta }\vspace{0.3cm} \\ 
&  & \hspace{0.5cm}\times \displaystyle\frac{1}{2^{\alpha +\beta +1}\Gamma
(n)\Gamma (n+\alpha )\Gamma (\beta +1)\Gamma (\beta +2)+N\Gamma (n+\beta
+1)\Gamma (n+\alpha +\beta +1)}\vspace{0.3cm} \\ 
& < & 0.%
\end{array}%
\end{equation*}

Next we will show the behavior of this zero,%
\begin{equation}
\begin{array}{lll}
u_{n} & = & \displaystyle-\frac{2(n+\beta +1)(n+\alpha +\beta +1)}{%
(2n+\alpha +\beta +1)^{2}}\frac{1+NK_{n-1}(-1,-1)}{1+NK_{n}(-1,-1)}\vspace{%
0.3cm} \\ 
&  & \displaystyle-\frac{2n(n+\alpha )}{(2n+\alpha +\beta +1)^{2}}\frac{%
1+NK_{n}(-1,-1)}{1+NK_{n-1}(-1,-1)}+\frac{(\alpha +\beta +1)(\beta -\alpha
+1)}{(2n+\alpha +\beta +1)^{2}}\vspace{0.3cm} \\ 
& = & \displaystyle\frac{-2(n+\beta +1)(n+\alpha +\beta +1)-2n\left(
n+\alpha \right) +(\alpha +\beta +1)(\beta -\alpha +1)}{(2n+\alpha +\beta
+1)^{2}}\vspace{0.3cm} \\ 
&  & \displaystyle+\frac{N\left( P_{n}^{\alpha ,\beta }(-1)\right) ^{2}\text{%
{\LARGE {/}}}\left\Vert P_{n}^{\alpha ,\beta }\right\Vert _{\mu }^{2}}{%
(2n+\alpha +\beta +1)^{2}}\left[ \frac{2\left( n+\beta +1\right) \left(
n+\alpha +\beta +1\right) }{1+NK_{n}\left( -1,-1\right) }-\frac{2n\left(
n+\alpha \right) }{1+NK_{n}\left( -1,-1\right) }\right] \vspace{0.3cm} \\ 
& = & \displaystyle\frac{-4n^{2}-4\left( \alpha +\beta +1\right) n-\left(
\alpha +\beta +1\right) ^{2}}{\left( 2n+\alpha +\beta +1\right) ^{2}}\vspace{%
0.3cm} \\ 
&  & \displaystyle+2N\frac{\left( P_{n}^{\alpha ,\beta }(-1)\right) ^{2}%
\text{{\LARGE {/}}}\left\Vert P_{n}^{\alpha ,\beta }\right\Vert _{\mu }^{2}}{%
(2n+\alpha +\beta +1)^{2}}\left[ \frac{\left( n+\beta +1\right) \left(
n+\alpha +\beta +1\right) }{1+NK_{n}\left( -1,-1\right) }-\frac{n\left(
n+\alpha \right) }{1+NK_{n}\left( -1,-1\right) }\right] \vspace{0.3cm} \\ 
&  & \displaystyle-1+2N\frac{\left( P_{n}^{\alpha ,\beta }(-1)\right) ^{2}%
\text{{\LARGE {/}}}\left\Vert P_{n}^{\alpha ,\beta }\right\Vert _{\mu }^{2}}{%
(2n+\alpha +\beta +1)^{2}}\left[ \frac{\left( n+\beta +1\right) \left(
n+\alpha +\beta +1\right) }{1+NK_{n}\left( -1,-1\right) }-\frac{n\left(
n+\alpha \right) }{1+NK_{n}\left( -1,-1\right) }\right] .\vspace{0.3cm}%
\end{array}
\label{AsintZero-01}
\end{equation}%
But from

\begin{equation*}
\frac{\left( P_{n}^{\alpha ,\beta }(x)\right) ^{2}}{\left\Vert P_{n}^{\alpha
,\beta }\right\Vert _{\mu }^{2}}=K_{n}(x,x)-K_{n-1}(x,x)
\end{equation*}%
and (\ref{K00J}) it easily follows that%
\begin{equation}
\frac{\left( P_{n}^{\alpha ,\beta }(x)\right) ^{2}}{\left\Vert P_{n}^{\alpha
,\beta }\right\Vert _{\mu }^{2}}=K_{n-1}(x,x)\left( \frac{\left( n+\alpha
+\beta +1\right) \left( n+\beta +1\right) }{n\left( n+\alpha \right) }%
-1\right) .  \label{AsintZero-02}
\end{equation}%
From (\ref{AsintZero-01}) and (\ref{AsintZero-02}) we get%
\begin{equation*}
\begin{array}{lll}
u_{n} & = & \displaystyle-1+2N\frac{\left( P_{n}^{\alpha ,\beta }(-1)\right)
^{2}\text{{\LARGE {/}}}\left\Vert P_{n}^{\alpha ,\beta }\right\Vert _{\mu
}^{2}}{(2n+\alpha +\beta +1)^{2}\left( 1+NK_{n}\left( -1,-1\right) \right) }%
\vspace{0.3cm} \\ 
&  & \displaystyle\hspace{0.5cm}\times \left[ \frac{\left( n+\alpha +\beta
+1\right) \left( n+\beta +1\right) }{n\left( n+\alpha \right) }-\frac{%
1+NK_{n}\left( -1,-1\right) }{1+NK_{n-1}\left( -1,-1\right) }\right] \vspace{%
0.3cm} \\ 
& = & \displaystyle-1+2N\frac{n(n+\alpha )\left( P_{n}^{\alpha ,\beta
}(-1)\right) ^{2}\text{{\LARGE {/}}}\left\Vert P_{n}^{\alpha ,\beta
}\right\Vert _{\mu }^{2}}{(2n+\alpha +\beta +1)^{2}\left( 1+NK_{n}\left(
-1,-1\right) \right) }\vspace{0.3cm} \\ 
&  & \displaystyle\hspace{0.5cm}\times \left[ \frac{\left( P_{n}^{\alpha
,\beta }(-1)\right) ^{2}\text{{\LARGE {/}}}\left\Vert P_{n}^{\alpha ,\beta
}\right\Vert _{\mu }^{2}}{K_{n-1}\left( -1,-1\right) \left( 1+NK_{n-1}\left(
-1,-1\right) \right) }\right] ,\vspace{0.3cm} \\ 
& > & -1.%
\end{array}%
\end{equation*}%
Moreover, using the asymptotics of $K_{n}\left( -1,-1\right) $, which
according with (\ref{K00J}) is%
\begin{equation*}
K_{n-1}(-1,-1)
\end{equation*}%
\begin{equation*}
\begin{array}{ll}
= & \displaystyle\frac{1}{2^{\alpha +\beta +2}\Gamma \left( \beta +1\right)
\Gamma \left( \beta +2\right) }\vspace{0.3cm} \\ 
& \displaystyle\times \frac{e^{-\left( n+\beta +1\right) }\left( n+\beta
+1\right) ^{\left( n+\beta +1-\frac{1}{2}\right) }e^{-\left( n+\alpha +\beta
+1\right) }\left( n+\alpha +\beta +1\right) ^{\left( n+\alpha +\beta +1-%
\frac{1}{2}\right) }}{e^{-n}n^{n-\frac{1}{2}}\left( n+\alpha \right)
^{n+\alpha -\frac{1}{2}}e^{-\left( n+\alpha \right) }}\vspace{0.3cm} \\ 
\sim & \displaystyle\frac{1}{2^{\alpha +\beta +2}\Gamma \left( \beta
+1\right) \Gamma \left( \beta +2\right) }n^{2\left( \beta +1\right) },%
\vspace{0.3cm}%
\end{array}%
\end{equation*}%
and after some calculations, we finally obtain 
\begin{equation*}
\begin{array}{lll}
u_{n} & = & \displaystyle-1+2N\frac{\left[ (2\beta +2)n+(\alpha +\beta +1)%
\right] ^{2}}{n(n+\alpha )\left( 2n+\alpha +\beta +1\right) ^{2}}\vspace{%
0.3cm} \\ 
&  & \displaystyle\hspace{0.5cm}\times \frac{K_{n-1}\left( -1,-1\right) }{%
\left( 1+NK_{n}\left( -1,-1\right) \right) \left( 1+NK_{n-1}\left(
-1,-1\right) \right) }\vspace{0.3cm} \\ 
& \sim & \displaystyle-1+\frac{2^{\alpha +\beta +2}\left( \beta +1\right) %
\left[ \Gamma \left( \beta +2\right) \right] ^{2}}{N}n^{-2\left( \beta
+2\right) }.\vspace{0.3cm}%
\end{array}%
\end{equation*}

\bigskip

The electrostatic interpretation means that the equilibrium position for the
zeros under the presence of an external potential%
\begin{equation*}
\text{ln}Q(x)+\text{ln }(1-x)^{\alpha +1}(1+x)^{\beta +2},
\end{equation*}%
where the first one is a short range potential corresponding to a unit
charge located at the zero of $Q(x)$ and the other one is a long range
potential associated with the weight function.

\bigskip

\textbf{2.} Now, we will consider $a<-1$. In this case, 
\begin{equation*}
d\mu ^{\ast }(x)=(x-a)(1-x)^{\alpha }(1+x)^{\beta }dx.
\end{equation*}
As a consequence,

\setcounter{prop}{5}

\begin{prop}
The structure relation \emph{(\ref{Struct-Rel})} for the above measure is%
\begin{equation*}
\phi (x)D(p_{n}^{\ast }(a;x))=A(x,n)p_{n}^{\ast }(a;x)+B(x,n)p_{n-1}^{\ast
}(a;x),
\end{equation*}%
where 
\begin{eqnarray*}
\phi (x) &=&(x-a)(1-x^{2}), \\
A(x,n) &=&a_{n+1}(x-\beta _{n})+b_{n+1}-\lambda _{n}a_{n}-(a_{n+1}\gamma
_{n}+\lambda _{n}b_{n})\frac{1}{\lambda _{n-1}}-1+x^{2}, \\
B(x,n) &=&\left( a_{n+1}(x-\beta _{n})+b_{n+1}-\lambda _{n}a_{n}\right) 
\frac{\gamma _{n}}{\lambda _{n-1}}+a_{n+1}\gamma _{n}+\lambda _{n}b_{n} \\
&&-(a_{n+1}\gamma _{n}+\lambda _{n}b_{n})\frac{x-\beta _{n-1}^{\ast }}{%
\lambda _{n-1}}.
\end{eqnarray*}
\end{prop}

\begin{proof}
Using the notation%
\begin{equation}
\lambda _{n}=\lambda _{n}^{\alpha ,\beta }(a)=\dfrac{P_{n+1}^{\alpha ,\beta
}\left( a\right) }{P_{n}^{\alpha ,\beta }\left( a\right) },  \label{ratios-3}
\end{equation}%
from (\ref{p1}) we get 
\begin{equation*}
\left( x-a\right) p_{n}^{\ast }(a;x)=P_{n+1}^{\alpha ,\beta }(x)-\lambda
_{n}P_{n}^{\alpha ,\beta }(x).
\end{equation*}%
Taking derivatives with respect to $x$ in both hand sides of the above
expression, and multiplying by $1-x^{2}$ in both members, we see that%
\begin{equation*}
\begin{array}{l}
(1-x^{2})p_{n}^{\ast }(a;x)+\left( x-a\right) (1-x^{2})D\left( p_{n}^{\ast
}(a;x)\right) \vspace{0.3cm} \\ 
=\displaystyle(1-x^{2})D\left( P_{n+1}^{\alpha ,\beta }(x)\right) -\lambda
_{n}(1-x^{2})D\left( P_{n}^{\alpha ,\beta }(x)\right) .%
\end{array}%
\end{equation*}%
From (\ref{StrctRelJac}), 
\begin{equation*}
(1-x^{2})D\left( P_{n}^{\alpha ,\beta }(x)\right) =a_{n}P_{n}^{\alpha ,\beta
}(x)+b_{n}P_{n-1}^{\alpha ,\beta }(x),
\end{equation*}%
where 
\begin{equation*}
\begin{array}{l}
a_{n}(x)=a_{n}^{\alpha ,\beta }(x)=\displaystyle\frac{-n[\beta -\alpha
+(2n+\alpha +\beta )x]}{2n+\alpha +\beta },\vspace{0.3cm} \\ 
b_{n}=b_{n}^{\alpha ,\beta }=\displaystyle\frac{4n(n+\alpha )(n+\beta
)(n+\alpha +\beta )}{(2n+\alpha +\beta )^{2}(2n+\alpha +\beta -1)}.%
\end{array}%
\end{equation*}%
Thus, 
\begin{equation*}
\begin{array}{l}
(1-x^{2})p_{n}^{\ast }(a;x)+\left( x-a\right) (1-x^{2})D\left( p_{n}^{\ast
}(a;x)\right) \vspace{0.3cm} \\ 
=a_{n+1}P_{n+1}^{\alpha ,\beta }(x)+\left( b_{n+1}-\lambda _{n}a_{n}\right)
P_{n}^{\alpha ,\beta }(x)-\lambda _{n}b_{n}P_{n-1}^{\alpha ,\beta }(x).%
\end{array}%
\end{equation*}%
Now, from the TTRR (\ref{RRJac}) for the classical Jacobi polynomials, we
get 
\begin{equation*}
\begin{array}{l}
(1-x^{2})p_{n}^{\ast }(a;x)+\left( x-a\right) (1-x^{2})D\left( p_{n}^{\ast
}(a;x)\right) \vspace{0.3cm} \\ 
=\left[ a_{n+1}(x-\beta _{n})+b_{n+1}-\lambda _{n}a_{n}\right] P_{n}^{\alpha
,\beta }(x)-(a_{n+1}\gamma _{n}+\lambda _{n}b_{n})P_{n-1}^{\alpha ,\beta
}(x).%
\end{array}%
\end{equation*}%
Now, from (\ref{pneqKnKn-1}) and (\ref{ratios-3}), 
\begin{equation*}
P_{n}^{\alpha ,\beta }(x)=p_{n}^{\ast }(a;x)-\frac{\gamma _{n}}{\lambda
_{n-1}}p_{n-1}^{\ast }(a;x),
\end{equation*}%
we obtain 
\begin{equation*}
\begin{array}{l}
(1-x^{2})p_{n}^{\ast }(a;x)+\left( x-a\right) (1-x^{2})D\left( p_{n}^{\ast
}(a;x)\right) \vspace{0.3cm} \\ 
\displaystyle=\left[ a_{n+1}(x-\beta _{n})+b_{n+1}-\lambda _{n}a_{n}\right] %
\left[ p_{n}^{\ast }(a;x)-\frac{\gamma _{n}}{\lambda _{n-1}}p_{n-1}^{\ast
}(a;x)\right] \vspace{0.3cm} \\ 
\displaystyle\hspace{0.5cm}-(a_{n+1}\gamma _{n}+\lambda _{n}b_{n})\left[
p_{n-1}^{\ast }(a;x)-\frac{\gamma _{n-1}}{\lambda _{n-2}}p_{n-2}^{\ast }(a;x)%
\right] \vspace{0.3cm} \\ 
\displaystyle=\left[ a_{n+1}(x-\beta _{n})+b_{n+1}-\lambda _{n}a_{n}\right]
p_{n}^{\ast }(a;x)\vspace{0.3cm} \\ 
\displaystyle\hspace{0.5cm}-\left[ \left( a_{n+1}(x-\beta
_{n})+b_{n+1}-\lambda _{n}a_{n}\right) \frac{\gamma _{n}}{\lambda _{n-1}}%
+a_{n+1}\gamma _{n}+\lambda _{n}b_{n}\right] p_{n-1}^{\ast }(a;x)\vspace{%
0.3cm} \\ 
\displaystyle\hspace{0.5cm}(a_{n+1}\gamma _{n}+\lambda _{n}b_{n})\frac{%
\gamma _{n-1}}{\lambda _{n-2}}p_{n-2}^{\ast }(a;x).%
\end{array}%
\end{equation*}%
Now, using the TTRR (\ref{TTRRK}) for monic kernels, 
\begin{equation*}
\begin{array}{l}
(1-x^{2})p_{n}^{\ast }(a;x)+\left( x-a\right) (1-x^{2})D\left( p_{n}^{\ast
}(a;x)\right) \vspace{0.3cm} \\ 
\displaystyle=\left[ a_{n+1}(x-\beta _{n})+b_{n+1}-\lambda _{n}a_{n}\right]
p_{n}^{\ast }(a;x)\vspace{0.3cm} \\ 
\displaystyle\hspace{0.5cm}-\left[ \left( a_{n+1}(x-\beta
_{n})+b_{n+1}-\lambda _{n}a_{n}\right) \frac{\gamma _{n}}{\lambda _{n-1}}%
+a_{n+1}\gamma _{n}+\lambda _{n}b_{n}\right] p_{n-1}^{\ast }(a;x)\vspace{%
0.3cm} \\ 
\displaystyle\hspace{0.5cm}(a_{n+1}\gamma _{n}+\lambda _{n}b_{n})\frac{%
\gamma _{n-1}}{\lambda _{n-2}}\frac{(x-\beta _{n-1}^{\ast })p_{n-1}^{\ast
}(a;x)-p_{n}^{\ast }(a;x)}{\gamma _{n-1}^{\ast }}.%
\end{array}%
\end{equation*}%
According to (\ref{coef-RRTTK}) and (\ref{ratios-3}), we obtain 
\begin{equation*}
\frac{\gamma _{n-1}^{\ast }}{\gamma _{n-1}}=\frac{\lambda _{n-1}}{\lambda
_{n-2}}.
\end{equation*}%
Therefore 
\begin{equation*}
\begin{array}{l}
(1-x^{2})p_{n}^{\ast }(a;x)+\left( x-a\right) (1-x^{2})D\left( p_{n}^{\ast
}(a;x)\right) \vspace{0.3cm} \\ 
\displaystyle=\left[ a_{n+1}(x-\beta _{n})+b_{n+1}-\lambda
_{n}a_{n}-(a_{n+1}\gamma _{n}+\lambda _{n}b_{n})\frac{1}{\lambda _{n-1}}%
\right] p_{n}^{\ast }(a;x)\vspace{0.3cm} \\ 
\displaystyle\hspace{0.5cm}-\left[ \left( a_{n+1}(x-\beta
_{n})+b_{n+1}-\lambda _{n}a_{n}\right) \frac{\gamma _{n}}{\lambda _{n-1}}%
+a_{n+1}\gamma _{n}+\lambda _{n}b_{n}\right. \vspace{0.3cm} \\ 
\hspace{5.5cm}\displaystyle\left. -(a_{n+1}\gamma _{n}+\lambda _{n}b_{n})%
\frac{x-\beta _{n-1}^{\ast }}{\lambda _{n-1}}\right] p_{n-1}^{\ast }(a;x).%
\end{array}%
\end{equation*}%
Thus 
\begin{equation*}
\phi (x)D(p_{n}^{\ast }(a;x))=A(x,n)p_{n}^{\ast }(a;x)+B(x,n)p_{n-1}^{\ast
}(a;x),
\end{equation*}%
where 
\begin{eqnarray*}
\phi (x) &=&(x-a)(1-x^{2}), \\
A(x,n) &=&a_{n+1}(x-\beta _{n})+b_{n+1}-\lambda _{n}a_{n}-(a_{n+1}\gamma
_{n}+\lambda _{n}b_{n})\frac{1}{\lambda _{n-1}}-1+x^{2}, \\
B(x,n) &=&\left( a_{n+1}(x-\beta _{n})+b_{n+1}-\lambda _{n}a_{n}\right) 
\frac{\gamma _{n}}{\lambda _{n-1}}+a_{n+1}\gamma _{n}+\lambda _{n}b_{n} \\
&&-(a_{n+1}\gamma _{n}+\lambda _{n}b_{n})\frac{x-\beta _{n-1}^{\ast }}{%
\lambda _{n-1}}.
\end{eqnarray*}%
Simplifying these expressions we have 
\begin{equation*}
\begin{array}{lll}
\displaystyle A(x,n) & = & A_{n,0}+A_{n,1}x+A_{n,2}x^{2}%
\end{array}%
\end{equation*}%
and 
\begin{equation*}
\begin{array}{lll}
B(x,n) & = & B_{n,0}+B_{n,1}x+B_{n,2}x^{2}.%
\end{array}%
\end{equation*}
\end{proof}

Notice that the Pearson equation for the linear functional associated with
the measure%
\begin{equation*}
d\mu ^{\ast }(x)=(x-a)(1-x)^{\alpha }(1+x)^{\beta }dx
\end{equation*}%
becomes%
\begin{equation*}
D\left( \phi u^{\ast }\right) =\psi u^{\ast },
\end{equation*}%
with (see (\ref{Pea1}) and (\ref{PearsonJac}))%
\begin{eqnarray*}
\phi (x) &=&(x-a)\sigma (x)=(x-a)(1-x^{2}), \\
\psi (x) &=&2\sigma (x)+(x-a)\tau (x)= \\
&&2(1-x^{2})+(x-a)(\beta -\alpha -(\alpha +\beta +2)x),
\end{eqnarray*}%
which means that in the Jacobi case, $Q\left( x\right) $ is the following
quadratic polynomial

\begin{equation*}
\begin{array}{lll}
Q\left( x\right) & = & \displaystyle B(x,n)+c_{n}\left[ -2A(x,n)+\phi
^{\prime }(x)-\psi (x)+\frac{c_{n}}{\gamma _{n-1}^{\ast }}B(x,n-1)\right] 
\vspace{0.3cm} \\ 
& = & \displaystyle\left[ B_{n,2}+\left( \alpha +\beta +1-2A_{n,2}+\frac{%
c_{n}B_{n-1,2}}{\gamma _{n-1}^{\ast }}\right) c_{n}\right] x^{2}\vspace{0.3cm%
} \\ 
&  & \displaystyle+\left\{ \frac{c_{n}^{2}B_{n-1,1}}{\gamma _{n-1}^{\ast }}%
+B_{n,1}-[\alpha (a-1)+\beta (a+1)+2A_{n,1}]c_{n}\right\} x\vspace{0.3cm} \\ 
&  & \displaystyle+B_{n,0}-[2A_{n,0}+1+a(\alpha -\beta )]c_{n}+\frac{%
c_{n}^{2}B_{n-1,0}}{\gamma _{n-1}^{\ast }}.%
\end{array}%
\end{equation*}

Taking into account%
\begin{eqnarray*}
\frac{\psi }{\phi } &=&\frac{2}{x-a}+\frac{\beta -\alpha -(\alpha +\beta +2)x%
}{1-x^{2}} \\
&=&\frac{2}{x-a}-\frac{\alpha +1}{1-x}+\frac{\beta +1}{1+x},
\end{eqnarray*}%
the electrostatic interpretation means that the equilibrium position for the
zeros under the presence of an external potential%
\begin{equation*}
\text{ln }Q(x)+\text{ln}\left( x-a\right) ^{2}(1-x)^{\alpha +1}(1+x)^{\beta
+1},
\end{equation*}%
where the first one is a short range potential corresponding to two unit
charges located at the zeros of $Q(x)$ and the second one is a long range
potential associated with a polynomial perturbation of the weight function.
\bigskip

\end{document}